\documentclass[letterpaper, 11pt]{amsart}

\usepackage{fullpage}
\usepackage{graphicx}
\usepackage{amsmath, amssymb, amsfonts, amsthm}

\usepackage{tikz} \usetikzlibrary{arrows}

\usepackage[pdftex,bookmarks,colorlinks,breaklinks]{hyperref}  
\usepackage{color}
\definecolor{dullmagenta}{rgb}{0.4,0,0.4}   
\definecolor{darkblue}{rgb}{0,0,0.4}
\hypersetup{linkcolor=red,citecolor=blue,filecolor=dullmagenta,urlcolor=darkblue} 

\newcommand{\x}{{\mathbf x}}
\newcommand{\y}{{\mathbf y}}
\newcommand{\vel}{{\mathbf v}}
\newtheorem{thm}{Theorem}[section]
\newtheorem{cor}[thm]{Corollary}
\newtheorem{prop}[thm]{Proposition}

\theoremstyle{remark}
\newtheorem{rem}[thm]{Remark}

\usepackage[backend=bibtex,style=alphabetic,natbib=true,url=false,sorting=nyt,doi=true,backref=false]{biblatex}
\renewbibmacro{in:}{\ifentrytype{article}{}{\printtext{\bibstring{in}\intitlepunct}}}

\begin{document}
\title{Spectrally Optimized \\ Pointset Configurations }

\author{Braxton Osting}
\address{Department of Mathematics, 
University of Utah, 
Salt Lake City, UT 84112, USA}
\email{osting@math.utah.edu }
\thanks{}

\author{Jeremy Marzuola}
\address{Department of Mathematics, University of North Carolina, Chapel Hill, NC  27599, USA}
\email{marzuola@math.unc.edu}
\thanks{}



\date{\today}

\begin{abstract}
The search for optimal configurations of pointsets, the most notable examples being the problems of Kepler and Thompson, have an extremely rich history with diverse applications in physics, chemistry, communication theory, and scientific computing. In this paper, we introduce and study a new optimality criteria for pointset configurations. Namely, we consider a certain weighted graph associated with a pointset configuration and seek configurations which minimize certain spectral properties of the adjacency matrix or graph Laplacian defined on this graph, subject to geometric constraints on the pointset configuration. This problem can be motivated by solar cell design and swarming models, and we consider several spectral functions with interesting interpretations such as spectral radius, algebraic connectivity,  effective resistance, and condition number. We prove that the regular simplex extremizes several spectral invariants on the sphere. We also consider pointset configurations on flat tori via (i) the analogous problem on lattices and (ii) through a variety of computational experiments.  For many of the objectives considered (but not all), the triangular lattice is extremal. 
\end{abstract}

\maketitle


\section{Introduction}
We consider \emph{optimal} configurations of pointsets in various spaces. Depending on the notion of optimality and the space considered, these types of problems have broad applicability in, for example, physics and  chemistry \cite{Torquato2010},  information theory and communication \cite{shannon2001,cohn2010}, and scientific computing \cite{matouvsek1999,damelin2003,borodachov2012}. Interesting and intriguing questions regarding the relationship between different notions of optimality criteria, equidistribution, and optimal geometric configurations abound; an overview of this subject can be found in the surveys of \cite{fejes1964,saff1997,ConwaySloane,cohn2010,Torquato2010}. 
In this paper, we define a new optimality criteria which is a variation of the typical problems considered. We are motivated by applications in solar cell design and models for swarming and agent interactions, although we expect that the criteria which we consider could have  other interesting  interpretations. In particular, we  consider a certain weighted graph associated with a pointset configuration and seek configurations which minimize certain spectral invariants  of the adjacency operator or graph Laplacian defined on this graph.  Before we describe the construction of the graph and our objective criteria, we review some related problems and their applications. 

\subsection*{Previous results on optimal pointset configurations} 
J. J. Thomson posed the problem of finding the steady-state distribution of $n$ identical charges constrained to the surface of a sphere \cite{thomson1904}. This is formulated as  finding configurations of points $\{ \x_i \}_{i\in[n]} \subset \mathbb S^2$ that minimize the Coulomb energy, 
$\displaystyle \sum_{\substack{i,j\in[n] \\ i\neq j}} \ | \x_i - \x_j |^{-1}. $
This problem has been generalized to include other types of interaction energies,
\begin{equation}
\label{eq:genPairEnergy}
V_f(\{ \x\}_{i=1}^n) = \sum_{\substack{i,j\in[n] \\ i\neq j}} \ f(| \x_i - \x_j |^2 ), 
\end{equation}
where $f$ is taken to be a suitable decreasing function. For example, the Riesz s-energy or Epstein zeta energy is given by  $f_s(r) = r^{-s/2}$,  the theta energy is given by  $f_\alpha(r) = e^{-2 \pi \alpha r}$ for any $\alpha>0$, and the logarithmic energy is given by $f(r) = - \log(r)$.  Thompson's problem can be recovered by taking $f(r) = r^{-\frac 1 2}$, which is of the form of the Riesz s-energy for $s=1$. This problem has also been generalized to finding an optimal pointset configuration on a variety of other manifolds, such as the torus. This family of problems and their applications was recognized by L. L. Whyte \cite{whyte1952unique} and  further gained popularity when Steve Smale listed the problem of finding the configuration to minimize the logarithmic energy as a mathematical problem for the next century  \cite{Smale1998}.  Recently, H. Cohn and A. Kumar \cite{cohn2007} considered optimal distributions of points on spheres for a broad class of energies. Using linear programming bounds, they showed that minimizers tend to  minimize a broad class of energies, a phenomena which they refer to as ``universal optimality''. In particular, the $E_8$ root system and the Leech lattice are universally optimal in 8 and 24 dimensions respectively and lower dimensional sections of these configurations can be shown to be universally optimal in dimensions 2--7 and 21--23.  In  \cite{ballinger2009}, a variety of computer experiments were conducted to find minimal energy spherical point configurations. In the dimensions where universal optimizers are known, these experiments can reproduce the optimal configurations. In the other dimensions, the structure of the optimal configurations and even the dimension of the optimizing space are still not well-understood.  

\subsection*{Dimension two and the triangular lattice} In dimension two, the triangular lattice configuration is a ubiquitous optimizer for pointset configurations and many of these properties are also realized in the closely related equilateral torus and honeycomb tilings. As we will show the triangular lattice is also optimal for several of the spectral invariants considered here, we briefly comment on several problems for which it is already known to be optimal. 

Optimal arrangements of spheres for a variety of  packing, kissing, and covering problems have  centers which are arranged in a triangular lattice  \cite{ConwaySloane}. For example, the maximum number of pennies that can be arranged to simultaneously touch a central one is six, obtained by the triangular configuration. Also, as shown by A. Thue, the triangular packing has the largest density of all packings in two dimensions.

Among lattices, the triangular lattice is the minimizer of  $V_f$ in \eqref{eq:genPairEnergy} (normalized) where $f$ is  (i) the Riesz s-energy or Epstein zeta energy,  $f(r) = r^{-s/2}$ for  any $s>0$ \cite{rankin1953,cassels1959,diananda1964,ennola1964} or (ii) the theta energy, $f_\alpha(r) = e^{-2 \pi \alpha r}$ for any $\alpha>0$ \cite{Montgomery1988}. This result has been extended in various ways for other potential energies. For completely monotonic functions $f$, the triangular lattice is the minimizer among equal volume Bravais lattices \cite{Betermin2015} and conjectured to be the universal minimizer among all configurations \cite{cohn2007}. 
We note however that the triangular lattice is not optimal for all potential energies; for example, the Lennard-Jones  energy can be tuned so that the triangular lattice is not optimal  \cite{Marcotte2011,betermin2014,Betermin2015}. 

The fact that the triangular lattice is optimal in so many contexts is not surprising when you consider its relatively large number of symmetries. These symmetries also manifest themselves into operators acting on functions defined on the lattice. For example, it is not difficult to show that the triangular lattice is the unique two-dimensional lattice for which the  nearest-neighbor finite-difference approximation of the Laplacian is isotropic at fourth order. In this paper, we will consider operators with matrix entries that depend smoothly on the pointset configuration. 

The equilateral torus also optimizes a variety of spectral quantities. M. Berger showed that the maximum first eigenvalue of the Laplace-Beltrami operator  over all flat tori of fixed volume  is attained only by the equilateral torus \cite{Berger1973,KLO2015}. For a broad class of Hilbert-Schmidt integral operators with an isotropic, stationary kernel, the equilateral torus maximizes the operator norm and the Hilbert-Schmidt norm  among all unit-volume flat tori \cite{OM1}.  Again however, the equilateral torus is not extremal for all spectral quantities; counterexamples exist involving heat kernels on flat tori \cite{baernstein1997}. 
 
\subsection*{Spectral invariants associated with a pointset}
We consider a finite or countably infinite pointset, $\{ \x_j \}_{j\in [n]}$, where each point $\x_j$ is distinct, {\it i.e.}, $d(\x_i, \x_j) > 0, \ \forall \ i\neq j$, and constrained to a compact  Riemannian manifold, $M$, {\it i.e.}, $\x_j \in M,  \ \forall \  j \in [n]$. Here the distance $d(\cdot, \cdot)$ is taken to be either the Euclidean distance if $M$ is embedded or the geodesic distance on $M$. In later sections, we will take M to be either a sphere $\mathbb S^{d-1}\subset \mathbb R^d$ or a $d$-dimensional torus, $\mathbb T^d$, but for now we only assume that $M$ is compact. 

We have in mind a process where the points in the set interact through a process which is dependent only on their pairwise distances. Thus, we define the following \emph{weighted adjacency matrix}, 
\begin{equation} \label{eq:W}
W^\x_{ij} = \begin{cases}  f \big( d^2( \x_i, \x_j) \big)  & i\neq j \\
0 &i = j 
\end{cases}
\end{equation}
where $f\colon (0,\infty) \to [0,\infty)$ is a smooth and nonnegative function. Note that, following the convention in \cite{cohn2007,cohn2010}, the function $f$ acts on the squared pairwise distances. 

For a given pointset, $\{ \x_j \}_{j\in [n]}$, we associate an undirected geometric graph, $G^\x =(V,E^\x)$.  We will abuse notation by using $V=[n]$ to denote both the indexing of the points and the vertex set for the graph. An edge is present between vertices $i, j \in [n]$ if $W^\x_{ij} > 0$.  The weight matrix $W^\x$ can be viewed as a collection of weights on $E^\x$ and we might also call $W^\x$ the \emph{weighted adjacency matrix} associated to the graph $G^\x$. 
The \emph{weighted degree} of vertex $i\in [n]$ is  defined
$d^\x_i = \sum_{j\in [n]} W^\x_{i,j}$. 
Let $D^\x$ denote the diagonal matrix defined by $(D^\x )_{ii} = d^\x_i$. The \emph{graph Laplacian}\footnote{In some contexts, $L^\x$ is referred to as the weighted unnormalized graph Laplacian \cite{von2007tutorial}, but we do not consider any of the other graph Laplacians here.} is then defined as
\begin{equation}
\label{eq:L}
L^\x = D^\x- W^\x. 
\end{equation} 

Generally speaking, a pointset, $\{ \x_j \}_{j\in [n]} \subset M$, interacting via the weighted adjacency matrix \eqref{eq:W} or graph Laplacian  \eqref{eq:L} will be limited by the spectrum of these operators. Thus,  we are motivated to study various spectral functions of these operators, which we will describe in detail in Section~\ref{sec:bg}. If we let  $\{\lambda_i^\x \}_{i=1}^n$ denote the eigenvalues of the graph Laplacian, examples of spectral quantities we will study are 
the trace $ \sum_{j = 1}^n \lambda_j^\x$, 
the effective resistance $\sum_{j \neq 1} (\lambda_j^\x )^{-1}$, 
the algebraic connectivity $\lambda_2^\x$, 
the spectral radius $\lambda_n^\x$, 
 the condition number $\lambda_n^\x/\lambda_2^\x$, 
 and distance between the spectrum and a given interval.

\subsection*{Results.} In this paper, we consider the above spectral invariants  for pointset configurations on spheres and tori. In Section \ref{sec:sphere}, we consider optimal configurations of $n\geq 3$ points on the sphere,  $\mathbb S^{n-2} \subset \mathbb R^{n-1}$. We prove that under certain assumptions on $f$, the regular simplex extremizes the trace,  spectral radius, total effective resistance, and the algebraic connectivity of the associated graph. The result for the trace of the graph Laplacian reduces to a well-known result for Thompson's problem \eqref{eq:genPairEnergy} and the result for the total effective resistance  follows {\it mutatis mutandis}. The proofs for the spectral norm and algebraic connectivity are generalized  to account for the fact that these functions are non-differentiable. 

We approach the problem of finding optimal pointset configurations on tori  by first considering  analogous problems for lattices in Section~\ref{sec:lat}. 
We prove a convergence result that makes precise the relationship between these problems. In Theorems~\ref{thm:LatMatNorm} and \ref{thm:LapMoments}, we prove that for certain $f$, the triangular lattice minimizes certain moments of the density of states for an operator defined on the lattice. 
 Finally, in Section~\ref{sec:tori}, we present some numerical results for optimal configurations on  flat tori. For a variety of spectral invariants and judicious choice of the function $f$, computational results provide evidence that the triangular lattice is extremal  among general (non-lattice) configurations.  However, in Section \ref{sec:interval}, for a specific choice of objective function, we find pointset configurations that are triangular lattices locally, but which have defects, and have smaller objective function values than the truncated triangular lattice.

\subsection{Motivations.} We are motivated by several applications which we describe below. The problems we consider here are abstract versions with varying fidelity to the original motivating problems. 

\subsection*{Motivation: Fermi's golden rule,  spectral band gaps, and solar cell design}
Our primary  motivation for studying optimal configurations of pointsets stem from a variety of problems in physics and engineering where it is of interest to either increase or decrease the density of states associated with a particular  frequency (or band of frequencies) by modifying an environmental variable. For example, in the quantum setting, it is sometimes desirable to enhance or inhibit spontaneous emission, as described by Fermi's Golden Rule, by modifying the environment of an atom \cite{CohenTannoudji,Kirr-Weinstein:03,Kirr-Weinstein:01,potDesign}. The idea of controlling the lifetime of states by varying the characteristics of a background potential goes back to the work of E. Purcell \cite{Purcell1952,Purcell:1946sf}, who reasoned that the lifetime of a state can be influenced by manipulating the set of states to which it can couple, and through which it can radiate. One approach to prevent propagation of waves at particular frequencies is through the use of periodicity to introduce spectral band gaps \cite{yablonovitch1987}. 
Another example of this type of problem arises in solar cell design where it is desirable to engineer a device to optimally harness solar energy, in the sense that there is maximal absorption of energy for a band of frequencies for Maxwell's equation related to the solar spectrum of light  \cite{yu2010fundamental,owen-survey,owen1,owen2}.  Such problems can be formulated as finding spatially-varying coefficients in Maxwell's equations so that there are scattering resonances near the band of the solar spectrum. 

Since scattering poles for Maxwell's equations are difficult to compute, one approach to simplifying this problem would be to consider the tight-binding approximation in which one obtains discrete operators similar to the weighted adjacency matrix \eqref{eq:W}. Since the spectrum of these graph operators are better understood and relatively inexpensive to compute, the spectral optimization problems are more accessible. However, as the computational methods we employ depend only on the computation of eigenvalues and their derivatives with respect to the locations of the points, with a fast numerical solver for Maxwell (or Helmholtz) scattering resonances a similar approach could be taken for the solar cell problem, which will be a topic for future work. 

\subsection*{Motivation: Swarming, flocking, and agent interaction models}
There are a variety of agent-based models which describe the time evolution of a system of ``agents'' that interact according to certain deterministic rules. When the agents align or self-organize in interesting ways, the behavior is termed \emph{swarming} or \emph{flocking}, see, {\it e.g.}, \cite{cucker2007,cucker2007b,motsch2011}. 
Let the particles have positions $\{\x_i \}_{i\in[n]}$ and velocities $\{\vel_i \}_{i\in[n]}$, and consider the time evolution equations 
\begin{align*}
\dot \x_i &= \vel_i \\ 
\dot \vel_i &= \sum_{j \neq i} a(\x_i,\x_j)   (\vel_j - \vel_i). 
\end{align*}
In  the Cucker-Smale model, the interaction kernel, $a(\cdot,\cdot)$, is symmetric and depends only on pairwise distances between particles, {\it i.e.}, 
$a(\x_i,\x_j) = f\left( d^2( \x_i, \x_j) \right)$. 
It follows that some long-time dynamics of the Cucker-Smale model can be inferred from the spectral properties of the adjacency matrix,  $(W^\x)_{ij} = f\left( d^2( \x_i, \x_j)\right)$ for arbitrary pointset configurations. Extremal pointset configurations can provide  bounds on such quantities.

\subsection*{Motivation: Low-dimensional spectral embedding of data}
In a variety of data analysis problems, it is of interest to reduce the dimension of a dataset by embedding into a relatively low dimensional space. This work can be interpreted as finding datasets that have extremal spectral embedding properties \cite{belkin2003laplacian,belkin2006manifold,singer2006graph}. 

\bigskip

\subsection*{Outline} In Section \ref{sec:bg}, we give some mathematical preliminaries. 
In Section \ref{sec:sphere}, we consider optimal configurations on spheres. 
In Section \ref{sec:lat}, we discuss optimal lattice configurations. 
In Section \ref{sec:tori}, we consider optimal pointset configurations on tori.  
We conclude in Section \ref{sec:disc} with a brief discussion.  
Sensitivity analysis of spectral quantities is discussed in Appendix~\ref{sec:sensitivity analysis} and parameterization of lattices in Appendix~\ref{sec:LattParam}.

 \subsection*{Acknowledgements} The authors would like to thank Mikael Rechtsman for starting them down the path of this work, as well as Laurent Betermin, Elena Cherkaev, Owen Miller, and Peter Mucha for valuable discussions along the way to completing it. We also thank the anonymous referees for many helpful suggestions. BO gratefully acknowledges support from NSF DMS-1461138.  JLM is supported by NSF Applied Math Grant DMS-1312874 and NSF CAREER Grant DMS-1352353.

\section{Preliminaries} \label{sec:bg} 
\subsection{Eigenvalues  of the weighted adjacency matrix associated to a pointset} \label{eq:eigW}
Here we discuss some basic properties of the eigenvalues of the weighted adjacency matrix, $W^\x$, 
defined in \eqref{eq:W}, associated to a pointset, $\{ \x_j \}$.
Let $\mu_i^\x$  for $i\in [n]$ denote the eigenvalues of $W^\x$. Since $W^\x$ is symmetric, its eigenvalues are real and characterized by the Courant minimax principle, 
\begin{equation*}
\mu_{i}^\x = \max_{C \in \mathbb R^{(i-1) \times n}} \  \min_{\substack{ \|v\|=1 \\ C v = 0}} \  \langle v, W^\x v \rangle .  
\end{equation*} 
The sum of the eigenvalues is given by the trace of $W^\x$,
$$
\sum_{i \in [n] } \mu_i^\x = \text{tr} W^\x = 0.
$$
We have that $\|W^\x\| \leq d_{+}^\x$ where $d^\x_+ := \max_{i\in [n]} d^\x_i $. It follows that all eigenvalues are contained in the interval $[- d_{+}^\x, d_{+}^\x]$. If $f$ is a positive definite function, then considerably more can be proven about the spectrum of $W^\x$ \cite{Wendland2005}. We do not assume $f$ to be positive definite here. 

\subsection{Eigenvalues  of the graph Laplacian associated to a pointset} \label{eq:eigGraphLap}
Here we discuss some basic properties of the eigenvalues of the graph Laplacian, $L^\x$, defined in \eqref{eq:L}, associated to a pointset,  $\{ \x_j \}_{j \in [n]}$. We also discuss the spectral quantities that will later be optimized. 

Let $\{ \x_j \}_{j \in [n]} \subset M$ where $[n]$ is the enumeration for a collection of $n$ points and $M$ is a compact Riemannian manifold.  
As $L^\x$, defined in \eqref{eq:L}, is the graph Laplacian for a graph with non-negative graph weights, 
the spectral properties of this matrix are well-studied \cite{Mohar:1991,Chung:1997,Biyikoglu:2007}. 
Let $\lambda_i^\x$  for $i\in [n]$ denote the eigenvalues of the graph Laplacian associated with a particular pointset configuration. Since $L^\x$ is symmetric, its eigenvalues are real and characterized by the Courant minimax principle, 
\begin{equation}
\label{eq:minimax}
\lambda_{i}^\x = \max_{C \in \mathbb R^{(i-1) \times n}} \  \min_{\substack{ \|v\|=1 \\ C v = 0}} \  \langle v, L^\x v \rangle .  
\end{equation} 
 The sum of the eigenvalues is given by the trace of the graph Laplacian 
\begin{equation}
\label{eq:trL}
\sum_{i\in [n]} \lambda_i^\x = \textrm{tr} L^\x = \sum_{i \in [n]} d_i^\x = \sum_{i,j \in [n]} W^\x_{i,j}. 
\end{equation}
This is precisely the quantity given in \eqref{eq:genPairEnergy}. 

The graph Laplacian $L^\x = D^\x - W^x$ has another decomposition, which reveals several additional   spectral properties. Let  $B\in \mathbb R^{\binom{n}{2} \times n}$ be the arc-vertex incidence matrix (a.k.a. graph gradient) for a complete directed graph $G=(V,E)$ on $|V| = n$ nodes,
\begin{equation*}
\label{eq:incidence}
B_{k,j} = \begin{cases}  
1 & j = \text{head} (k) \\
-1 &j = \text{tail} (k)  \\
0 & \text{otherwise}.
\end{cases}
\end{equation*}
Here, we have used the terminology that if an arc $k=(i,j)$ is directed from node $i$ to node $j$ then $i$ is the \emph{tail} and $j$ is the \emph{head} of  arc $k$. The arc orientations (heads and tails of arcs) can be chosen arbitrarily; we use the convention that for edge $\{i,j\}$,  $i = \text{head}(k)$ and $j=\text{tail}(k)$ if $i>j$. 
We construct a weight vector $w^\x \in \mathbb R^{\binom n 2}_+$ by 
\begin{equation}
\label{eq:vecw}
w^\x_k := f\big(d^2(\x_i, \x_j) \big) = W_{ij}^\x 
\qquad \text{for edge} \ k = \{i,j\}, 
\end{equation}
where the weight matrix, $W^\x$, is defined in  \eqref{eq:W}.
The graph Laplacian can then be decomposed  
\begin{equation}
\label{eq:divgrad}
L^\x = B^{t} \ \text{diag}(w^\x) \ B,
\end{equation}
where $\text{diag}(w^\x)$ is a diagonal matrix with entries given by $w^\x$. 
From this div-grad decomposition, it is easily  seen that the inner product in \eqref{eq:minimax} can be rewritten 
$\langle v, L^\x v \rangle = \| B v \|^2_{w^\x}$, where $\| f \|^2_{w^\x} := \sum_{k\in E} w_k f_k^2$. 
It follows that all eigenvalues are contained in the interval $[0, 2 d_{+}^\x]$, where $d^\x_+ := \max_{i\in V} d^\x_i $. The first eigenvalue  $\lambda_{1}^\x$, is zero with corresponding eigenvector $v_{1}=1$. 
The second eigenvalue, $\lambda_{2}^\x$, is nonzero if and only if the graph is connected and characterized by
\begin{equation}
\label{eq:lam2}
\lambda_{2}^\x = \min_{\substack{ \|v\|=1 \\ \langle v, 1 \rangle = 0}} \ \| B v \|^2_{w^\x} .  
\end{equation}
The second eigenvalue is also referred to as the \emph{algebraic connectivity} as it is closely related to the other notions of connectivity of a graph \cite{Fiedler:1973,Ghosh:2006b,Ghosh:2006}. This graph invariant arises in the analysis of a variety of graph processes that describe, for example, information transfer rates for dynamical models
\cite{Bjorner:1991,Olfati-Saber:2007,Sun2004,boydDiaconisXiao2004}, robustness and stability in
inverse problems \cite{boumal2012,OBO2012,Osting:2012b} and synchronizability in complex networks \cite{Arenas2008}. It is also widely used in graph partitioning and data clustering algorithms \cite{shi2000,von2007tutorial} due to its close relationship to the Cheeger constant.
As a function of the graph weights $w^\x$, the algebraic connectivity is non-decreasing and concave. This can be seen from \eqref{eq:lam2} since $w^\x \mapsto \lambda_{2}^\x$ is the pointwise minimum of a family of  functions, each of which is linear in $w^\x$. 

The \emph{spectral radius} $\lambda^\x_{\text{max}}= \lambda^\x_n$, is characterized by 
\begin{equation}
\label{eq:LSpecRad}
\lambda_{n}^\x = \max_{ \|v\|=1 } \  \| B v \|^2_{w^\x} .  
\end{equation}
As a function of the graph weights $w^\x$, the spectral radius is a convex and non-decreasing function. 
This  can be seen from \eqref{eq:LSpecRad} since $w^\x \mapsto \lambda_{n}^\x$ is the pointwise maximum of a family of  functions, each of which is linear in $w^\x$. 

The \emph{total effective resistance} of the graph $G^\x$ associated with pointset $\{\x_i \}_{i=1}^n$  is defined 
\begin{equation}
\label{eq:EffRes}
R_{\mathrm{tot}}^\x := n \sum_{i\neq 1} \frac{1}{\lambda_i^\x}  = n \cdot \mathrm{tr} (L^\x)^\dag,
\end{equation}
where $\cdot^\dag$ denotes the Moore-Penrose pseudoinverse. As a function of the graph weights $w^\x$, the total  effective resistance is monotone decreasing and convex \cite{ghosh2008}.  The total effective resistance arises in the analysis of electrical networks, as well as in other applications involving Markov chains and continuous-time averaging networks. 
 
The \emph{condition number} of the graph Laplacian $L^\x$ associated with pointset $\{\x_i \}$  is defined 
 \begin{equation}
\label{eq:CondNum}
\kappa^\x := \lambda_n^\x/ \lambda_2^\x. 
\end{equation}
As a function of the graph weights, $w^\x$, the condition number is quasi-convex, {\it i.e.} has convex level sets. This follows from the observation that 
$$\kappa(w^\x) \leq \alpha  \iff  \lambda_n(w^\x) - \alpha \lambda_2(w^\x) \leq 0
$$
and $\lambda_n(w)$ and $-\lambda_2(w)$ are convex functions. 

\subsection{Notation for the dependence of a spectral invariant on a pointset configuration} \label{sec:Notation}
We consider the general  optimization problem  of minimizing a spectral invariant with respect to a  pointset configuration on a compact manifold $M$, 
\begin{equation}
\label{eqref:genOptProb}
\min \Big\{ g\left(\{ \x_i \}_{i\in[n]} \right) \colon \{ \x_i \}_{i\in[n]} \subset M \Big\}  . 
\end{equation}
It will be convenient to introduce some notation so that $g\colon M^n \to \mathbb R$ is the composition of functions
$$
g = \Phi \circ F \circ D^2. 
$$
Here, 
\begin{equation} \label{eq:Phi}
\Phi = J \circ \lambda \circ L \colon \mathbb R^{\binom n 2} \to \mathbb R
\end{equation}
 where $J\circ \lambda$ is a spectral function and $L\colon \mathbb R^{\binom n 2} \to {\bf S}^n$ maps graph weights to the weighted graph Laplacian, where notationally we have taken ${\bf S}^n$ to be the set of $n \times n$ real, symmetric matrices.
$J$ and $\lambda$ are further described in Appendix \ref{sec:sensitivity analysis}. 
 $F\colon\mathbb R^{\binom n 2} \to \mathbb R^{\binom n 2}$ defined  by $F_j( v) = f(v_j)$ is a function that applies element-wise a function $f \colon \mathbb R \to \mathbb R$, and  $D^2\colon M ^n  \to \mathbb R^{\binom n 2}$ is the vector of all of the squared distances of a pointset configuration $\{ \x_i \}_{i=1}^n \subset M$. The sensitivity of spectral invariants with respect to the pointset configuration is described in Appendix~\ref{sec:sensitivity analysis}. 

\section{Spectral optimal pointset configurations on spheres}  \label{sec:sphere}
In this section, we consider the problem of finding spectrally optimal pointset configurations on a sphere, $\mathbb S^{n-2} \subset \mathbb R^{n-1}$, 
\begin{equation}
\label{eq:genSphere}
\min  \Big\{  g\left(\{ \x_i \}_{i\in[n]} \right) \colon \{ \x_i \}_{i=1}^n \subset \mathbb S^{n-2} \Big\}. 
\end{equation}
For a general  number of points this is a difficult problem, but for $n\geq 3$ points on $\mathbb S^{n-2}$, there are a number of spectral objective functions where we can show that the optimal pointset configuration is the regular simplex.

The following proposition states that the regular simplex  is the minimizer for $ \mathrm{tr} L^\x$. Since  $\mathrm{tr} L^\x = \sum_{ i\neq j} \ f(| \x_i - \x_j |^2 )$ this problem reduces to the generalized Thompson problem \eqref{eq:genPairEnergy}. Although this result is well-known (see, {\it e.g.}, \cite{cohn2010}), 
we include a proof for completeness and also because it provides a template for the other spectral objective functions considered. We use the notation introduced in Section~\ref{sec:Notation}.

\begin{prop} \label{prop:SphereTrace} Let $f \colon (0,4] \to \mathbb R$ be any differentiable, decreasing, and convex function. Then the regular simplex attains the minimum in \eqref{eq:genSphere} with $g\left(\{ \x_i \}_{i\in[n]} \right) = \mathrm{tr} L^\x$, as defined in \eqref{eq:trL}. 
\end{prop}
\begin{proof}
We first note that for any points, $\x,\y\in \mathbb S^{n-2}$, the Euclidian  distance is given by 
$|\x-\y|^2 = 2 - 2 \langle \x, \y \rangle $.
We also compute for $\x_i \in \mathbb S^{n-2}$, 
$\displaystyle 0 \leq \left| \sum_{i\in[n]} \x_i\right|^2 = n +  \sum^n_{i \neq j } \langle \x_i, \x_j \rangle $,
which implies that 
\begin{equation} \label{eq:sphereFact}
 \sum^n_{i> j} \langle \x_i, \x_j \rangle \geq - \frac{n}{2}.
\end{equation}
Note that $\Phi(w) = \mathrm{tr}L(w)$ is linear in each argument and therefore $\Phi  \circ F$ is a convex function since it is a positive linear combination of convex functions. Recall from Section \ref{sec:Notation} that $F$ is a function that applies element-wise the function $f$. 
 Thus, for any vectors $a,b \in \mathbb R^{\binom {n} 2}$, we have 
$$
\Phi \circ F(a) \geq \Phi \circ F(b) + \langle a-b , (\Phi \circ F)'(b) \rangle.
$$
We now take $a=d^2$ to be the vector of squared pairwise distances for an arbitrary configuration and $b$ to be the vector of squared pairwise distances for the regular simplex, $b = 2 + \frac{2}{n-1} = \frac{2n}{n-1} $. We compute
\begin{align*}
 \Phi \circ F(d^2) 
&\geq \Phi \circ F\left( \frac{2n}{n-1} \right ) + \Big\langle d^2 - \left(2 + \frac{2}{n-1} \right) , \nabla (\Phi \circ F) \left( \frac{2n}{n-1} \right) \Big \rangle \\
& = \binom{n}{2} f\left( \frac{2n}{n-1}  \right) + \sum_{i> j}^{n} \left( 2 \langle \x_i, \x_j \rangle + \frac{2}{n-1} \right) \left| f'\left(  \frac{2n }{n-1}\right) \right| \\
& \geq  \binom{n}{2} f\left( \frac{2n}{n-1}  \right),
\end{align*}
which is the value attained by the regular simplex. 
\end{proof}

\begin{rem}
The squared Frobenius norm of the adjacency matrix is given by 
$$
\| A^\x \|_F^2 = \sum_{i\neq j} f(d_{ij}^2)^2.
$$
Thus, if $f$ is chosen so that $f^2$ is a differentiable, decreasing, and convex function,  it follows from Proposition~\eqref{prop:SphereTrace}  that the regular simplex attains $\min \big\{  \| A^\x \|_F^2 \colon \{\x_i\}_{i\in[n]} \subset \mathbb S^{n-2}  \big\}$. 
\end{rem}

\begin{prop}  \label{prop:SphereEffRes} Let $f \colon (0,4] \to \mathbb R$ be any differentiable, increasing, and concave function. 
 Then the regular simplex attains the minimum in \eqref{eq:genSphere} with $g\left(\{ \x_i \}_{i\in[n]} \right) = R_{\mathrm{tot}}^\x $,  as defined in \eqref{eq:EffRes}. 
\end{prop}
\begin{proof} Let $\Phi = R_{tot}\colon \mathbb R^{\binom n 2} \to \mathbb R$ be the total effective resistance of a graph as a function of the edge weights. The total effective resistance is a convex and monotone decreasing function in each argument. 
We can express $R_{\mathrm{tot}}^\x$ in \eqref{eq:EffRes} as $R_{\mathrm{tot}}^\x = R_{tot} \circ F \circ D^2$ 
where $F$ and $D$ are as in Section \ref{sec:Notation}. 
Thus if $f$ is concave, the composition $R_{\mathrm{tot}} \circ F$ is convex. Let $d^2$ be a vector of squared pairwise distances for an arbitrary configuration. We compute
\begin{align*}
 \Phi \circ F(d^2) 
&\geq \Phi \circ F\left( \frac{2n}{n-1} \right ) -2 \langle r ,g \rangle 
\end{align*}
where  $r_\ell = \langle \x_i, \x_j \rangle + \frac{1}{n-1}$ where $\ell = (i,j)$ and 
$g = \nabla  (\Phi \circ F) (\frac{2n}{n-1})$.  
By Proposition~\ref{prop:objFunDiffL}, the gradient is a constant vector, which can also be seen from symmetry. Since $f$ is increasing and $R_{\mathrm{tot}}$ is decreasing, the constant is negative. By \eqref{eq:sphereFact}, we have that $\sum_\ell r_\ell \geq 0$ which implies that $- 
\langle r, g  \rangle \geq 0$. We conclude that 
$ \Phi \circ F(d^2)  \geq  \Phi \circ F \left( \frac{2n}{n-1}  \right)$,
which is the value attained by the regular simplex. 
\end{proof}

\begin{prop} \label{prop:SphereMaxEig} Let $f \colon (0,4] \to \mathbb R$ be any differentiable, decreasing, and convex function. 
Then the regular simplex attains the minimum in \eqref{eq:genSphere}
with $g\left(\{ \x_i \}_{i\in[n]} \right) = \lambda_{n} (L^\x)$, as defined in \eqref{eq:LSpecRad}.
\end{prop}
\begin{proof} 
The spectral function $\Phi (w) = \lambda_n(w)$ is convex and non-decreasing in each argument. Since $f$ is assumed to be convex, it follows that $\Phi\circ F$ is a convex function. 
Let $\partial f$ denote the subdifferential  of the function $f$. Recall that the \emph{subdifferential} of $f\colon \mathbb R^n \to \mathbb R$ at the point $\bar x$ is the set-valued map given by 
$$
\partial f(\bar x) = \{ \phi \in \mathbb R^n \colon \langle \phi, x-\bar x \rangle \leq f(x) - f(\bar x) \text{ for all } x\in \mathbb R^n \}; 
$$
see, for example, \cite{BoLe2006}. As in the proof of Proposition~\ref{prop:SphereTrace}, let $d^2$ to be the vector of squared pairwise distances for an arbitrary configuration. For any subderivative $g \in \partial ( \Phi \circ F) \left( \frac{2n}{n-1} \right)$, we have
\begin{subequations}
\label{eq:LamMax}
\begin{align}
 \Phi \circ F(d^2) 
&\geq \Phi \circ F\left( \frac{2n}{n-1} \right ) + \Big\langle d^2 - \left(2 + \frac{2}{n-1} \right) , g \Big \rangle \\
&\geq \Phi \circ F\left( \frac{2n}{n-1} \right ) + 2 \left| f'\left(\frac{2n}{n-1}\right) \right|  f\left( \frac{2n}{n-1} \right) \langle r , h  \rangle 
\end{align}
\end{subequations}
Here $h \in \partial \Phi(1)$ and $r \in \mathbb R^{\binom n 2}$ with components $r_\ell = \langle x_i, x_j \rangle + \frac{1}{n-1}$ where $\ell = (i,j)$. We also used the fact that 
$$\partial ( \Phi \circ F) \left( \frac{2n}{n-1} \right) = - \left| f'\left(\frac{2n}{n-1}\right) \right|  f\left( \frac{2n}{n-1} \right)   \partial \Phi(1).$$ 
It remains to show that there exists $h\in \partial \Phi(1)$ such that $\langle r, h \rangle \geq 0$. We note that $1 \notin \partial \Phi(1)$ (otherwise  the proof could be reduced to the proof of  Proposition~\ref{prop:SphereTrace}). 

It is not difficult to show for any $\psi \in \mathbb R^n$ satisfying $\langle \psi, 1\rangle = 0$ and $\|\psi \|=1$, we have 
$$h = (B\psi)^2 \in \partial \Phi(1).$$
Here $\cdot^2$ should be interpreted as an element wise operation. It follows that
\begin{subequations}
\label{eq:maxrh}
\begin{align}
\max_{h \in \partial \Phi(1)} \langle r, h \rangle 
&= \max_{\substack{\|\psi\|=1 \\ \langle \psi,1 \rangle = 0}} \langle r, (B \psi)^2 \rangle \\
\label{eq:maxrhb}
&= \max_{\substack{\|\psi\|=1 \\ \langle \psi,1 \rangle = 0}} \sum_{i>j} \langle x_i, x_j \rangle ( \psi_i - \psi_j)^2 + \frac{1}{n-1} \sum_{i>j} ( \psi_i - \psi_j )^2\\
\label{eq:maxrhc}
& = \mu_{max} + \frac{n}{n-1}.
\end{align}
\end{subequations}
The second term in \eqref{eq:maxrhb} simplifies because $\langle \psi, 1 \rangle=0$ implies that $B^t B \psi = n \psi$. 
The first term in \eqref{eq:maxrhb} can be viewed as the largest eigenvalue of the matrix $B^t \text{diag}(\omega) B$ where $\omega = r -  \frac{1}{n-1}$, which we denote by $\mu_{max}$ in \eqref{eq:maxrhc}. Letting $\mu_i$ for $i=1,\ldots,n$ be the eigenvalues of $B^t \text{diag}(\omega) B$, and noting that at least one of the eigenvalues is zero, we  compute
$$\mu_{max} \geq \frac{1}{n-1} \sum_{\mu_i \neq 0} \mu_i 
= \frac{1}{n-1} \sum_{i} \mu_i 
= \frac{1}{n-1} \text{tr} \left( B^t \text{diag}(\omega) B \right) 
= \frac{2}{n-1} \sum_{i>j} \langle x_i, x_j \rangle  
\geq - \frac{n}{n-1},
$$
where the last line follows from \eqref{eq:sphereFact}. We have  shown that there exists $\hat h \in \partial \Phi(1)$ attaining the maximum in  \eqref{eq:maxrh} with $\langle r, \hat h \rangle \geq 0$. The result now follows from \eqref{eq:LamMax}.
 \end{proof}

\begin{prop} Let $f \colon (0,4] \to \mathbb R$ be any differentiable, decreasing, and concave function. 
Then the regular simplex attains the maximum in \eqref{eq:genSphere} with $g\left(\{ \x_i \}_{i\in[n]} \right) = \lambda_{2} (L^\x)$, as defined in \eqref{eq:lam2}.
\end{prop}
\begin{proof}
The algebraic connectivity, $\lambda_2 \big( L(w) \big)$ is a non-decreasing and concave function of the graph weights $w$. Since $f$ is assumed to be concave, it follows that $\Phi  \circ F$ is a concave function. 
Let $d^2$ be a vector of squared pairwise distances for an arbitrary configuration. 
For any superderivative $g\in \{ (B \psi)^2 \colon L \psi = \lambda_2 \psi  \}$, we have
\begin{align*}
 \Phi \circ F(d^2) 
&\leq
 \Phi \circ F \left(\frac{2n}{n-1}  \right) - 2 \langle r,g  \rangle 
\end{align*}
where $r_\ell = \langle \x_i, \x_j \rangle + \frac{1}{n-1}$ for $\ell = (i,j)$. An analogous argument to that given in the proof of Proposition~\ref{prop:SphereMaxEig} shows that $- \langle r,g  \rangle \leq 0$. We conclude that 
$ \Phi \circ F(d^2)  \leq \Phi \circ F \left( \frac{2n}{n-1}  \right)$, which is the value attained by the regular simplex. 
\end{proof}

\section{Spectrally optimal lattices}  \label{sec:lat}
In this section, we discuss spectral properties of operators associated with Bravais lattices. These results are used in Section~\ref{sec:tori} for pointset configurations on flat tori.  

Let $\Lambda =B(\mathbb Z^d)$ denote the $d$-dimensional Bravais lattice with basis $B\in \mathbb R^{d\times d}$. The reciprocal (dual) lattice, $\Lambda^* = 2 \pi B^{-t}(\mathbb Z^d)$, consists of the set of vectors, $\xi$, such that $e^{\imath v\cdot \xi} = 1$ for every $v \in \Lambda$. The Brillouin zone, denoted $\mathcal B \subset \mathbb R^d$, is defined as the Voronoi  cell\footnote{The Voronoi cell is also sometimes referred to as the Dirichlet cell or Wigner-Seitz cell.} of the origin in the dual lattice. For $\psi\in \ell^2(\Lambda)$, the discrete Fourier transform and its inverse are defined 
\begin{align*}
&\hat \psi (\xi) = \mathcal F [\psi] (\xi) = \sum_{v\in \Lambda} e^{\imath \xi\cdot v} \psi(v) 
= \frac{1}{2} \sum_{v\in \Lambda} \cos( \xi\cdot v) \psi(v),  \qquad  \xi \in \mathcal B \\
&\psi(v) = \mathcal F^{-1} [\hat \psi] (v) = \frac{1}{|\mathcal B|} \int_{\mathcal B}  e^{-\imath \xi \cdot v} \hat{\psi}(\xi) \ d\xi,  \qquad  v \in \Lambda.
\end{align*}
Here, $|\mathcal B| = \left| \det\left(  2 \pi B^{-t}\right) \right| = (2 \pi)^d \left| \det B \right|^{-1}$ denotes the volume of the Brillouin zone, $\mathcal B$. We restrict our attention to two-dimensional lattices ($d=2$).

Let $f\colon \mathbb R\to \mathbb R$ be a non-negative function with sufficiently fast decay so that 
\begin{equation} \label{eq:fsum}
\sum_{u \in \Lambda\setminus \{ 0 \}}  f(\| u\|^2) < \infty.
\end{equation}
We consider the linear operator 
$ W_f$, defined by
$$
(W_f \psi) (v) := \sum_{u \in \Lambda\setminus \{v\}} f(\| u - v\|^2) \psi(u), \qquad v\in \Lambda.
$$
Note that by \eqref{eq:fsum},  $ W_f \colon \ell^2 (\Lambda) \to \ell^2(\Lambda)$. 
In what follows, since $f (r) $ is never evaluated at $r=0$, we assume $f(0) = 0$, so the sum can be taken over $u \in \Lambda$. 
For $u,v\in \Lambda$, this operator has  ``matrix elements''   $W_f(u,v) = f( \| u-v \|^2)$.  Note that $W_f$ is analogous the matrix $W^\x$, defined in \eqref{eq:W}. 
Observing that $W_f(u,v) = W_f(v,u)$ and 
$W_f(u+v,u) = W_f(v,0)$, we see that the operator is symmetric and acts by convolution. We also observe that 
$\text{tr}(W_f) = \sum_{u\in \Lambda} f (0) = 0$. 
Define the operator symbol (dispersion relation)
\begin{equation}
\label{eq:dispRel}
\omega_f(\xi) := \mathcal F[f( \| \cdot \|^2)](\xi), \qquad \qquad \xi \in \mathcal B. 
\end{equation}
In Figure~\ref{fig:dispRel}, we plot $\omega_f(\xi)$ for the square and triangular two-dimensional  lattices, with $f(r) = e^{-2r}$ for $r > 0$.  Note, since we are considering  We have that 
$$\omega_f(\xi) = \overline{\omega_f(\xi)} = \omega_f(-\xi)$$ 
implying that $\omega_f(\xi)$ is real and has an inversion symmetry. 
We also have that for all $\xi \in \mathcal B$,
$$
\omega_f(\xi) =  \sum_{u \in \Lambda} e^{\imath \xi \cdot u} f(\| u\|^2) 
\leq  \sum_{u \in \Lambda}  f(\| u\|^2)
= \omega_f(0),
$$
which shows that  $\omega_f$ attains its maximum at the origin. 

We compute 
\begin{align*}
W_f e^{\imath \xi \cdot v} = \sum_{u \in \Lambda} f(\|u-v \|^2) e^{\imath \xi \cdot u} 
= \sum_{w \in \Lambda}  f(\|w \|^2) e^{\imath \xi \cdot (v+w)} 
= \omega_f(\xi) e^{\imath \xi \cdot v}. 
\end{align*}
This shows that $W_F$ is diagonalized by the discrete Fourier transform, 
\begin{equation} \label{eq:AFdiag}
(W_f \psi )(v) = \mathcal F^{-1} \omega_f \mathcal F \psi 
= \frac{1}{|\mathcal B|} \int_{\mathcal B}  e^{-\imath \xi \cdot v} \omega_f (\xi) \hat{\psi}(\xi) \ d\xi.
\end{equation}
Note that $W_f$ is not a compact operator since  $e^{\imath \xi \cdot v} \notin \ell^2(\Lambda)$. 
By  Plancherel's theorem, we have that 
$$
\|W_f \psi \|_{\ell^2(\Lambda)} 
\leq \left(  \max_{\xi \in \mathcal B} \ | \omega_f(\xi) | \right) \| \psi \|_{\ell^2(\Lambda)} 
= \omega_f(0) \| \psi \|_{\ell^2(\Lambda)}, 
$$
which implies
\begin{equation}
\label{eq:normAf}
 \|W_f \|_{\ell^2(\Lambda) \to \ell^2(\Lambda)} = \omega_f(0) .
\end{equation}
Thus, $W_f  \colon \ell^2 (\Lambda) \to \ell^2(\Lambda)$ is a symmetric bounded  linear operator. We note that $W_f$ is a self-adjoint operator if $f$ has compact support \cite{MoharWoess1989}. 
The associated quadratic form
$$
\psi \mapsto \langle \psi,  W_f \psi \rangle = \int_{\mathcal B} \omega_f(\xi) | \hat \psi(\xi) |^2 \ d\xi
$$
is not  positive if $\omega_f(\xi)$ is not positive on $\mathcal B$. There are also conditions on $f$ which imply that $W_f+f(0)  \mathrm{Id}$ is a positive definite operator \cite{Wendland2005}. 

\bigskip

We also define the linear operator 
$ L_f  \colon \ell^2 (\Lambda) \to \ell^2(\Lambda)$, 
by
$L_f := D_f - W_f$, 
where $D_f$ is the operator which is just multiplication by the scalar $ \sum_{u\in \Lambda} f(\| u\|^2 ) = \omega_f(0)$. 
We refer to $L_f$ as the \emph{Laplacian} operator on the graph. It is the lattice analogue of the matrix $L^\x$, defined in \eqref{eq:L}. Since $D_f$ is a diagonal operator, the Laplacian is diagonalized by the discrete Fourier Transform, $L_f = \mathcal F^{-1} (\omega_f(0) - \omega_f) \mathcal F$. 
The associated quadratic form is given by 
$$
\psi \mapsto \langle \psi,  L_f \psi \rangle = \int_{\mathcal B} (\omega_f(0) - \omega_f(\xi))  | \hat \psi(\xi) |^2 \ d\xi. 
$$
Since $\omega_f(\xi)$ takes its maximum at the origin, the Laplacian is a semi-positive definite operator. 

\begin{figure}[t]
\begin{center}
\includegraphics[height=.37\linewidth]{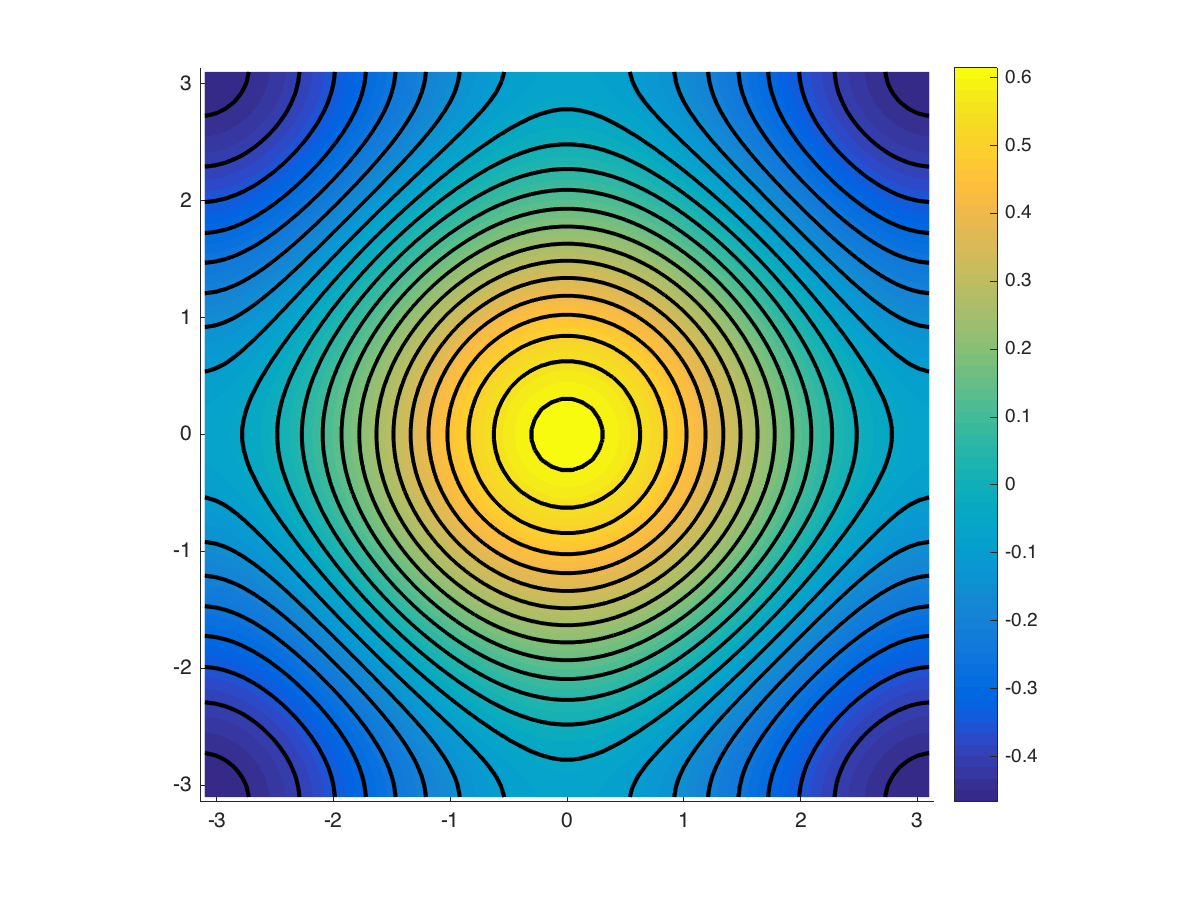}
\includegraphics[height=.37\linewidth]{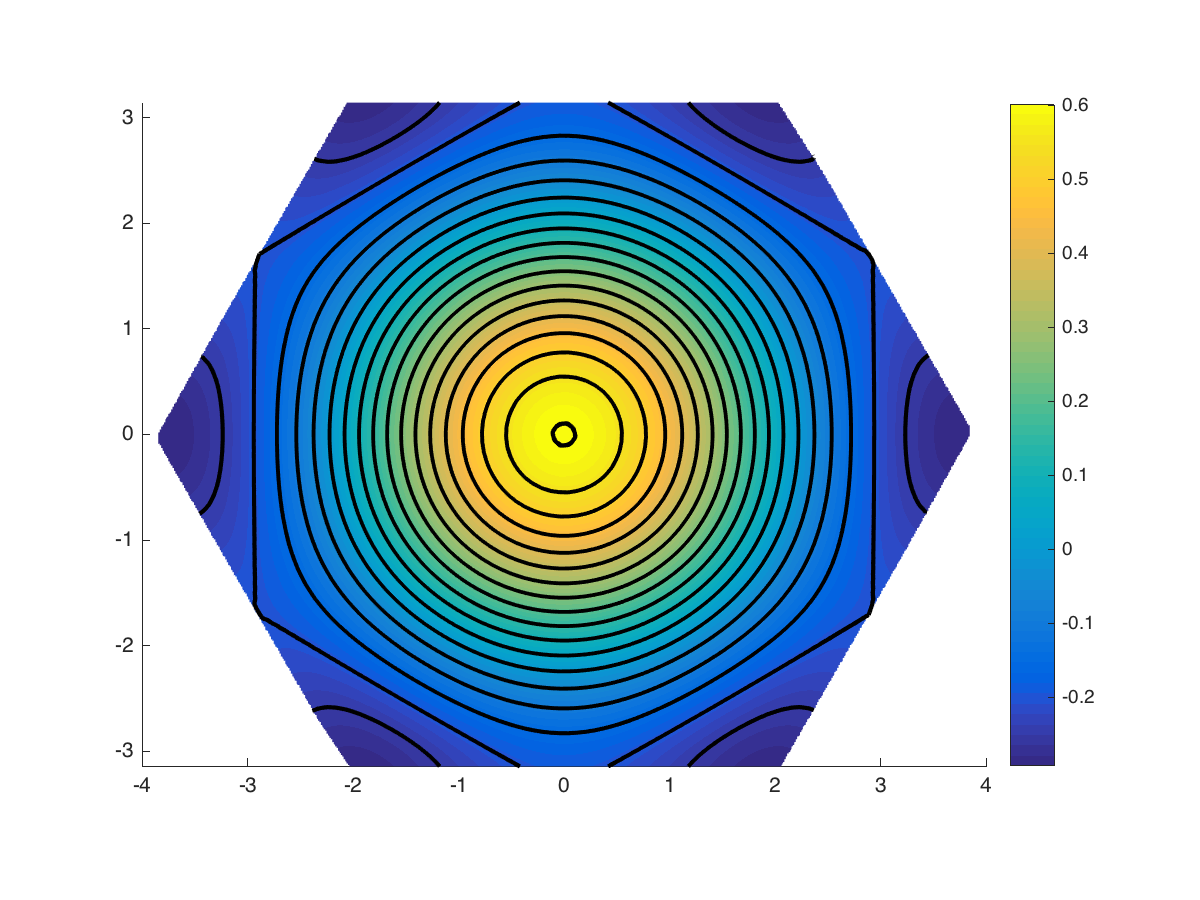}
\caption{The dispersion relation, $\omega_f(\xi)$, with $f(r) = e^{-2r}$, for the square lattice (left) and triangular lattice (right). The black lines indicate  level sets of $\omega_f(\xi)$. }
\label{fig:dispRel}
\end{center}
\end{figure}

\subsection*{Spectral decomposition and density  of $W_f$.}  
Let $\sigma$ denote the spectrum of $W_f$ and   $\sigma \ni \lambda = \omega_f (\xi)$. We interpret $\omega_f^{-1}(\lambda) \subset \mathcal B$ as the set of wavenumbers corresponding to frequency $\lambda$.  
From \eqref{eq:AFdiag} and the coarea formula, we arrive at the spectral decomposition of $W_f$, 
\begin{align*}
(W_f \psi)(v) &= \frac{1}{|\mathcal B|} \int_{\mathcal B}  e^{-\imath \xi \cdot v} \omega_f (\xi) \hat{\psi}(\xi) \ d\xi \\
& =\frac{1}{|\mathcal B|}  \int_{ \sigma }  \left[    \int_{\omega_f^{-1}(\lambda)}  e^{-\imath \xi \cdot v}  \omega_f (\xi) \hat{\psi}(\xi) \frac{1}{| \nabla \omega_f |} \ dH(\xi)  \right]  \ d\lambda \\
&=\int_{\sigma } \lambda \left[  \frac{1}{|\mathcal B|}  \int_{\omega_f^{-1}(\lambda)}  e^{-\imath \xi \cdot v}  \sum_{u\in \Lambda} e^{\imath \xi \cdot u} \psi(u)   \frac{1}{| \nabla \omega_f |} \ d H(\xi)  \right] \ d\lambda \\
& = \int_\sigma \lambda \  dE_\lambda[\psi](v)
\end{align*}
Here $H$ denotes the one-dimensional Hausdorff measure. 
The projection valued measure associated with $W_f$  is  defined by
\begin{equation} \label{eq:pvm} 
 \psi \mapsto dE_\lambda[\psi](v) := \left[  \frac{1}{|\mathcal B|}  \int_{\omega_f^{-1}(\lambda)}  e^{-\imath \xi \cdot v}  \sum_{u\in \Lambda} e^{\imath \xi \cdot u} \psi(u)   \frac{1}{| \nabla \omega_f |} \ dH(\xi)  \right] \ d\lambda. 
\end{equation}
The (Plancharel)  spectral measure for $W_f$ is absolutely continuous and given by
$$
h_f(\lambda_0) = - \frac{1}{\pi}  \Im \lim_{\epsilon \to 0+}  \int_{\sigma } (\lambda_0  - \lambda - \imath \epsilon)^{-1}  \left( \int_{\omega_f^{-1}(\lambda)}   \frac{1}{| \nabla \omega_f |}  \ dH(\xi) \right) d \lambda . 
$$
By Sokhatsky's formula, we have 
\begin{equation}
\label{eq:DOS}
h_f(\lambda) = \chi_\sigma(\lambda)    \int_{\omega_f^{-1}(\lambda)}  \frac{1}{| \nabla \omega_f |}  \ d H(\xi)  , 
\end{equation}
where $\chi_\sigma(\lambda)$ denotes the characteristic function on the spectrum of $W_f$. We think of $h_f(\lambda) d\lambda$ as giving a measure of the ``number of states'' in the frequency interval $[\lambda, \lambda + d\lambda]$ and, consequentially, in the physics literature, $h_f(\lambda)$ is referred to as the \emph{density of states} \cite{Economou1984,Lin2016}.  Roughly speaking, a high density of states at a specific frequency interval means that there are many states available for occupation.  In various applications it is useful to engineer a device which has either a large or small density of states at a particular frequency or in a particular frequency interval.

\begin{figure}[t]
\begin{center}
\includegraphics[width=.6\linewidth]{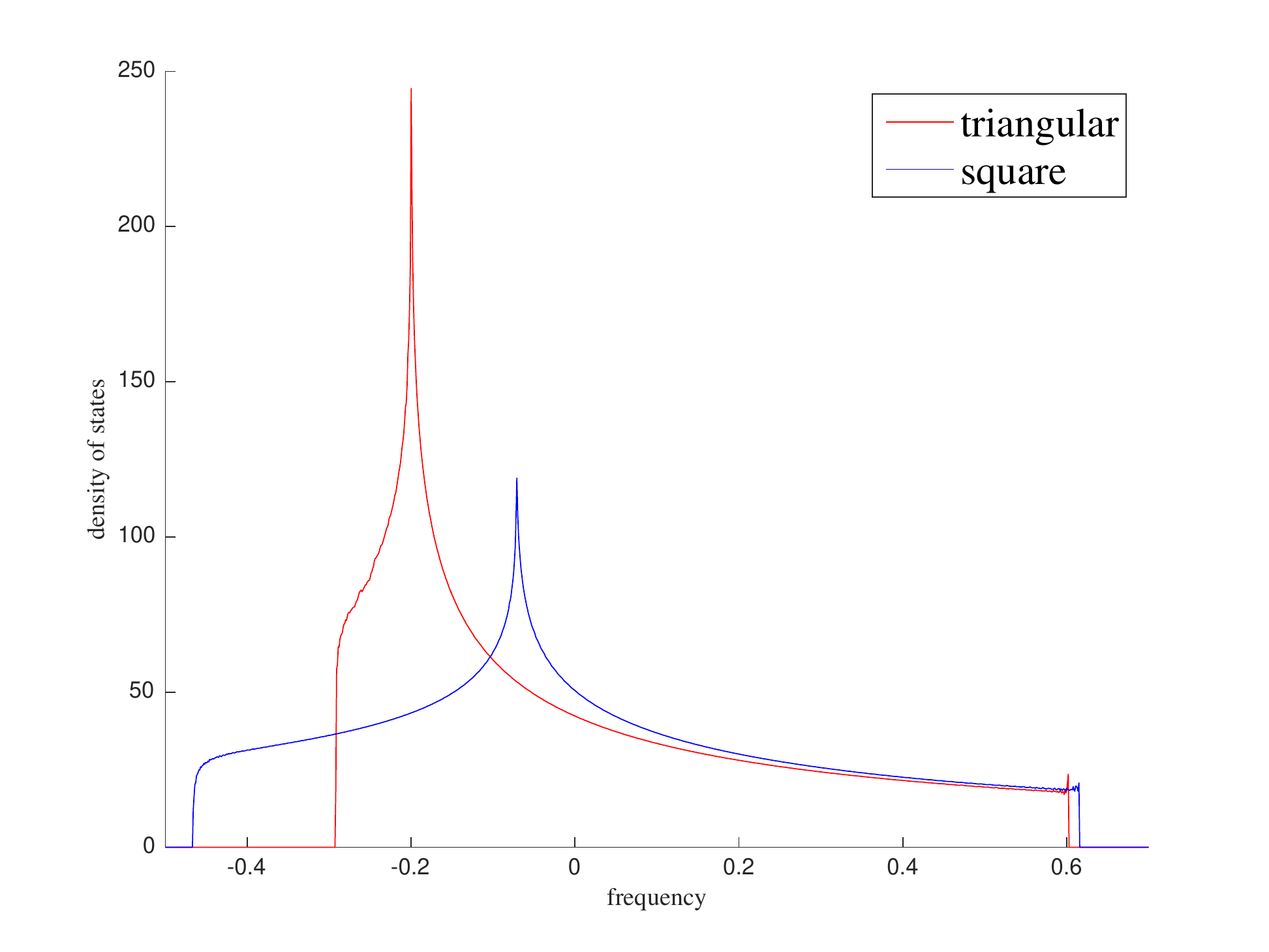}
\caption{A comparison of the density of states, $h_f(\lambda)$, defined in \eqref{eq:DOS}, with $f(r) = e^{-2r}$, for the (volume normalized) square and  triangular lattices. The dispersion relations, $\omega_f$, for these two lattices are plotted in Figure~\ref{fig:dispRel}.  }
\label{fig:spectralDensityTriVsSq}
\end{center}
\end{figure}

In Figure~\ref{fig:dispRel}, the black curves represent level sets of $\omega_f(\xi)$ for $\xi \in \mathcal B$.  In Figure~\ref{fig:spectralDensityTriVsSq}, we plot the spectral densities, $h_f(\lambda)$ with $f(r) = e^{-2r}$, for the square (blue) and  triangular (red) lattices. 
For this choice of function $f$, we observe Van Hove logarithmic singularities at certain values of the density of states: $\lambda \approx -0.2$ for the triangular lattice and $\lambda \approx -0.05$ for the square lattice \cite{VanHove1953,Economou1984}. 
The logarithm singularity is integrable and occurs at the  largest value $\lambda$ such that $\omega_f^{-1}(\lambda)\subset \mathcal B$ intersects $\partial \mathcal B$. 

It is natural to study the $p$-th moment of the  density of states, defined
\begin{equation} \label{eq:pMomAf}
M_{W_f}^p := \int_{\mathbb R} \lambda^p \ h_f(\lambda) \ d\lambda = \int_{\mathcal B} \omega^p_f(\xi) \ d\xi. 
\end{equation}
The first few moments are computed as follows. For $p=0$, we simply obtain 
$M^0_{W_f} = \int_{\mathcal B} \ d\xi = |\mathcal B|$. 
The first moment (mean) is given by 
\begin{align*}
M^1_{W_f} 
&= \langle 1 , \omega_f 1 \rangle_{L^2 (\mathcal B)} \\
&= \langle 1 , \mathcal F W_f {\mathcal F}^{-1} 1 \rangle_{L^2 (\mathcal B)}\\ 
&= |\mathcal B|  \langle \delta_0 , W_f \delta_0 \rangle_{\ell^2(\Lambda)} \\
&= |\mathcal B|  \langle \delta_0 , f(\| \cdot \|^2) \rangle_{\ell^2(\Lambda)} \\
&= 0,
\end{align*}
where we used the facts that $\mathcal F[\delta_0] (\xi) = 1$ and $\langle g, \hat \psi \rangle_{L^2(\mathcal B)} = |\mathcal B| \langle  \check g, \psi \rangle_{\ell^2(\Lambda)}$ for $g\in L^2(\mathcal B)$ and $\psi \in \ell^2(\Lambda)$. 
The first moment can be interpreted as a (normalized) trace of $W_f$. 
The second moment can be computed
\begin{subequations} \label{eq:2ndMomAf}
\begin{align}
M^2_{W_f} 
&= \langle 1 , \omega^2_f 1 \rangle_{L^2 (\mathcal B)} \\
&= \langle 1 , \mathcal F A^2_f {\mathcal F}^{-1} 1 \rangle_{L^2 (\mathcal B)} \\
&= |\mathcal B|  \langle W_f \delta_0 , W_f \delta_0 \rangle_{\ell^2(\Lambda)} \\
&= |\mathcal B|  \langle f(\| \cdot \|^2) , f(\| \cdot \|^2) \rangle_{\ell^2(\Lambda)} \\
&= |\mathcal B| \sum_{u \in \Lambda \setminus \{0\}} f^2 (\| u\|^2).
\end{align}
\end{subequations}

\subsection*{Spectral decomposition and density  of $L_f$.}  

The spectral decomposition for the Laplacian can be written 
$$
(L_f \psi) (v) = \int_{\sigma'} (\omega_f(0) - \lambda) \ dE_\lambda[\psi](v),
$$
where $\sigma'$ is the spectrum of the Laplacian and the projection valued measure is given in \eqref{eq:pvm}. Note that $\sigma ' = \omega_f(0) - \sigma$, where  $\sigma$ is the spectrum of $W_f$. 
The density of states at frequency $\lambda$  is given by $h_f(\omega_f(0)-\lambda)$ where $h_f$ is defined in \eqref{eq:DOS}. The $p$-th moment of the density of states associated with  the Laplacian is defined 
\begin{equation*} \label{eq:pMomLf}
M_{L_f}^p := \int_{\mathbb R} (\omega_f(0) - \lambda)^p \ h_f(\lambda) \ d\lambda = \int_{\mathcal B} \left(\omega_f(0) - \omega_f(\xi) \right)^p \ d\xi. 
\end{equation*}
The zeroth moment is computed $M^0_{L_f} = |\mathcal B|$. The first moment is given by 
\begin{equation} \label{eq:1stMomLf}
M^1_{L_f} = |\mathcal B |  \omega_f(0) = |\mathcal B | \sum_{u\in \Lambda \setminus \{0\}} f(\| u\|^2 ) . 
\end{equation}
As above, we interpret the first moment as a regularized trace of the Laplacian.

\subsection{Minimal  spectral invariants  over lattices }
In this section, we address the question of minimizing spectral invariants over lattices for fixed $f$. We first consider the operator norm $\|W_f \|_{\ell^2(\Lambda) \to \ell^2(\Lambda)} $ and second moment $M_{W_f}^2$. 
\begin{thm} \label{thm:LatMatNorm}
If $f$ be a completely monotonic function satisfying \eqref{eq:fsum} for every unit-volume Bravais lattice, then the triangular lattice is the unique minimizer of $\|W_f \|_{\ell^2(\Lambda) \to \ell^2(\Lambda)} $ among all unit-volume Bravais lattices. If $f^2$ is a completely monotonic function satisfying 
$ \sum_{u \in \Lambda\setminus \{ 0 \}}  f^2(\| u\|^2) < \infty$ for every all unit-volume Bravais lattice, $\Lambda$,  then the triangular lattice is the unique minimizer of $M_{W_f}^2$ among all unit-volume Bravais lattices. 
\end{thm}
\begin{proof}
From  \eqref{eq:normAf}, we have then have that 
$$
\|W_f \|_{\ell^2(\Lambda) \to \ell^2(\Lambda)} = \omega_f(0) = \sum_{u \in \Lambda}  f(\| u\|^2),
$$ 
where $\omega_f$ is defined in \eqref{eq:dispRel}. From \eqref{eq:2ndMomAf}, we have that 
$$
M_{W_f}^2 = \sum_{u \in \Lambda}  f^2(\| u\|^2) .
$$
The result now follow from \cite[Prop. 3.1]{Betermin2015}.  In the special cases where $f(r)= e^{-\alpha r}$ with $\alpha >0$, the objective reduces to the theta function and the result  follows from the work  of H. L. Montgomery \cite[Theorem 1]{Montgomery1988}. In the special case with $f(r) = r^{-s}$  with $s > 1$, the objective reduces to the Epstein zeta function and the result follows from the work of   \cite{rankin1953,cassels1959,diananda1964,ennola1964}. 
\end{proof}

For large $p$, the $p$-th moment, as defined in \eqref{eq:pMomAf}, is dominated where $\omega_f(\xi)$ is large. The following result then follows from \eqref{eq:normAf} and Theorem \ref{thm:LatMatNorm}. 
 
\begin{cor} Let $f$ be a completely monotonic function satisfying \eqref{eq:fsum} for every unit-volume Bravais lattice. For $p$ sufficiently large, the triangular lattice is the unique minimizer of 
$M_{W_f}^p$ among all unit-volume Bravais lattices. 
\end{cor}

\bigskip

We now consider the first moment of the Laplacian for fixed $f$. 
\begin{thm} \label{thm:LapMoments}
Let $f$ be a completely monotonic function satisfying \eqref{eq:fsum} for every unit-volume Bravais lattice. The triangular lattice is the unique minimizer of $M_{L_f}^1$ 
among all unit-volume Bravais lattices.
\end{thm}
\begin{proof}
The proof follows from \eqref{eq:1stMomLf} and  \cite[Prop. 3.1]{Betermin2015}. 
\end{proof}

Thus, the triangular lattice has the smallest regularized trace of $L_f$. However, it is not obvious what the minimizer is for other spectral invariants, such as 
$$
\|L_f \|_{\ell^2(\Lambda) \to \ell^2(\Lambda)} = \omega_f(0) - \min_{\xi \in \mathcal B} \ \omega_f(\xi).
$$
To consider such spectral invariants, in the following section, we describe sequences of finite graphs with graph operators that converge (in the sense of spectral measure) to $W_f$ and $L_f$.

\subsection{Finite approximations to lattice configurations} \label{sec:ToriGraph}
In this section, we assume $f$ has compact support. We make this assumption because the graph associated with any pointset configuration with pairwise distances bounded below will be locally finite, {\it i.e.} the degree of every vertex is finite.
Let  $\Lambda = B(\mathbb Z^2)$ be a Bravais lattice. To approximate $\Lambda$, for any even integer $N$, we consider the $N^2$-vertex torus graph, 
\begin{equation} \label{eq:FiniteToriGraph}
\Lambda^N := \left\{ B x \colon x = (i / N - N/2), j/N - N/2) \text{ for } i,j \in \{0,1,\ldots, N-1 \}  \right\} .
\end{equation}
Note that this is equivalent to truncating a lattice and identifying the ``boundaries''. 
Define the following operator $W_f^N \colon \ell^2(\Lambda^N) \to \ell^2(\Lambda^N)$ by 
$$
W_f^N \psi(v) = \sum_{u \in \Lambda^N \setminus \{0\}} f\left( d^2(u,v) \right) \psi(u), \qquad v \in \Lambda^N. 
$$
Here $d$ should be interpreted as the distance after identification of the boundaries. 

We now study the limit  $N\to \infty$. To compare the operator $W_f^N \colon \ell^2(\Lambda^N) \to \ell^2(\Lambda^N)$ to the operator $W_f \colon \ell^2(\Lambda) \to \ell^2(\Lambda)$, we extend $W_f^N$ to $\ell^2(\Lambda)$, by 
$$
W_f^N \psi (v) = 
\begin{cases}
W_f^N \psi(v) & v\in \Lambda^N \\
0 & v \notin \Lambda^N.
\end{cases}
$$
Similarly, we define the operator $L_f^N \colon \ell^2(\Lambda^N) \to \ell^2(\Lambda^N)$ by 
$$
L_f^N \psi(v) = \left[ \sum_{u \in \Lambda^N \setminus\{0 \}} f\left( d^2(u,0) \right) \right] \psi(v)  - W^N_f  \psi(v), \qquad v \in \Lambda^N
$$
and extend $L_f^N$ to $\ell^2(\Lambda)$, by
$$
L_f^N \psi (v) = 
\begin{cases}
L_f^N \psi(v) & v\in \Lambda^N \\
0 & v \notin \Lambda^N.
\end{cases}
$$

Let us recall some statements from \cite{MoharWoess1989}, encapsulated there as Theorems $4.12$ and$4.13$.  See also the treatment of \cite{godsil1988walk}, Theorems $4.1$ and $4.2$, which focus specifically on spectral measures for infinite graphs.  In these works, as a means of proving information about the spectral measure of an infinite graph of bounded degree, convergence to this measure from a family of finite graph operators is discussed.  The statements are essentially that given a finite family of sub-graphs converging to an infinite graph with bounded degree, then the spectral measures at each lattice point converge to the spectral measure at a lattice point for the infinite graph, and the spectral radii of the finite graphs converge to that of the infinite.  To prove this theorem, Lemma $4.11$ in  \cite{MoharWoess1989} is applied, which is a classical theorem from functional analysis stating that given a sequence of self-adjoint operators acting on $\ell^2$ that converge weakly (namely pointwise for each element of $\ell^2$), then the resolutions of identity $\mu_n ( (-\infty, \lambda)$ converge to $\mu ( (-\infty, \lambda)$, the spectral measure of the limit for all $\lambda$ at which $\mu$ is continuous.  

\begin{prop} Assume  $f\colon \mathbb R\to \mathbb R$ is a non-negative, compactly supported function satisfying \eqref{eq:fsum}. 
In the strong operator topology, $W_f ^N \to W_f$ and $L_f^N \to L_f$ as $N\to \infty$ . 
\end{prop}

\begin{proof}

 Let $\psi \in \ell^2(\Lambda)$ be compactly supported. Since $f$ has compact support, for sufficiently large $N$, 
 $W_f^N \psi = W_f \psi$. By the density of compactly supported functions in $\ell^2(\Lambda)$, $W_f ^N \to W_f$. 
Since $f$ has compact support, the operators  $W^N_f$ and $L^N_f$  (and $W_f$ and $L_f$) are simply related by an inversion and finite shift, so $L_f^N \to L_f$ as well. 
 \end{proof}

Given a family of operators that thus converge in this sense, 
let $\mu^N$ and $\mu$ denote the resolutions of the identity for $W_f^N$ and $W_f$. By \cite[Lemma 4.11]{MoharWoess1989}, we have that $\mu_N((-\infty,\lambda))$ converges to $\mu((-\infty,\lambda))$ for every $\lambda$. In particular, we have for $\epsilon >0$,  
\begin{equation*} \label{eq:ApproxSpectralMeasure} 
\lim_{N\to \infty} \ 
 \frac{ 1 }{  N^2  } \cdot 
 \left( \text{number of eigenvalues of  $W_f^N$ in $[\lambda,\lambda + \epsilon]$} \right) 
\ = \
 \frac{1}{|\mathcal B|} \int_{\lambda}^{\lambda+\epsilon} h_f(\lambda) d\lambda. 
\end{equation*}
Thus, we will perform computational experiments to understand how some spectral invariants depend on the lattice by approximating lattices $\Lambda$ by finite tori graphs $\Lambda^N$, which is well-motivated  in the large lattice limit by the above analysis.  In addition, we note that while the analysis in this section applies to $f$ with compact support, we expect that the methods carry through for $f$ with sufficient decay.

\section{Spectral optimal pointset configurations on tori}  \label{sec:tori}
In Section \ref{sec:lat}, we considered lattice configurations and in Theorems \ref{thm:LatMatNorm} and \ref{thm:LapMoments} we proved that among lattices, the triangular lattice extremizes certain spectral invariants. In this section, we conduct a variety of computational experiments to extend  these results in two ways:
(i)  We consider a variety of spectral invariants that are more difficult to address analytically. 
(ii) Instead of just considering lattice configurations, in this section, we also consider general pointset configurations on flat tori.
To address these questions, we will present two types of numerical results. 
\begin{enumerate}
\item The first will be based on the results in Section \ref{sec:ToriGraph}. We consider the adjacency matrices  for  finite torus graphs defined in \eqref{eq:FiniteToriGraph}, which are periodic truncations of Bravais lattices. We parameterize the Bravais lattices as described in Appendix~\ref{sec:LattParam}  and plot the value of a spectral invariant over the set $U$ in Proposition~\ref{prop:LatticeParam} (see Figure~\ref{fig:FundDom}). 
\item We consider general pointset configurations on tori and use gradient-based optimization methods to find locally optimal pointset configurations for a spectral invariant.  
\end{enumerate}
The computational results of this section will support conjectures that the triangular lattice extremizes a variety of spectral invariants among general pointset configurations. 

This section is organized as follows. We begin with the trace of the graph Laplacian and squared Frobenius norm of the graph Adjacency matrix, since these spectral quantities reduce to the pairwise potential energy objective in \eqref{eq:genPairEnergy}. We then consider the spectral radius, total effective resistance, algebraic connectivity, and condition number, as defined in Section~\ref{sec:bg}.

\subsection{Spectral invariants reducing to the generalized Thompson's problem} \label{sec:Thompson}
We have already noted that a number of spectral invariants  give the generalized Thompson's problem, 
\begin{equation}
\label{eq:traceComp}
g\left(\{ \x_i \}_{i\in[n]} \right) = \sum_{\substack{i,j=1 \\ i\neq j}}^n f\left( d^2(\x_i, \x_j) \right) ;
\end{equation}
see \eqref{eq:genPairEnergy}. 
These include
(i) the second moment of the adjacency matrix spectrum, 
$$
\sum_k \mu_k^2(W^\x) = \mathrm{tr} (W^\x W^\x) = \| W^\x \|^2_F = \sum_{i,j} (W^\x_{i,j})^2 
$$
and 
(ii) the trace of the graph Laplacian, $\mathrm{tr}(L^\x) = \sum_k \lambda_k(L^\x)$, as defined in \eqref{eq:trL}.  
Note that these two quantities  are equivalent for different choices of function $f$. 
From Theorems~\ref{thm:LatMatNorm} and  \ref{thm:LapMoments}, we anticipate that if 
$f(r)$ be given by either $f(r)= e^{-\alpha r}$ with $\alpha >0$ or $f(r) = r^{-s}$  with $s > 1$, 
the triangular lattice configuration will be optimal, at least among lattice configurations. 

\medskip

We consider Bravais lattice configurations of points, $\Lambda = B(\mathbb Z^2)$ and the finite approximation $\Lambda^N$ as defined in \eqref{eq:FiniteToriGraph}. Let $N=10$ and $f\colon\mathbb R\to\mathbb R$ be the convex and decreasing function $f(r) = \exp(- \alpha r)$ with $\alpha = 2$.  In Figure~\ref{fig:traceFundDom}(left), we plot the trace of the graph Laplacian for the pointset configuration $\{\x_i\}_{i\in[N^2]} = \Lambda^N =  \Lambda^N_{a,b}$ as lattice parameters $(a,b)$ vary over the set $U$ in Proposition~\ref{prop:LatticeParam}.  We observe that the triangular lattice is a global minimum and the square lattice is a saddle point. 

\begin{figure}[t]
\begin{center}
 \includegraphics[height=.3\textwidth]{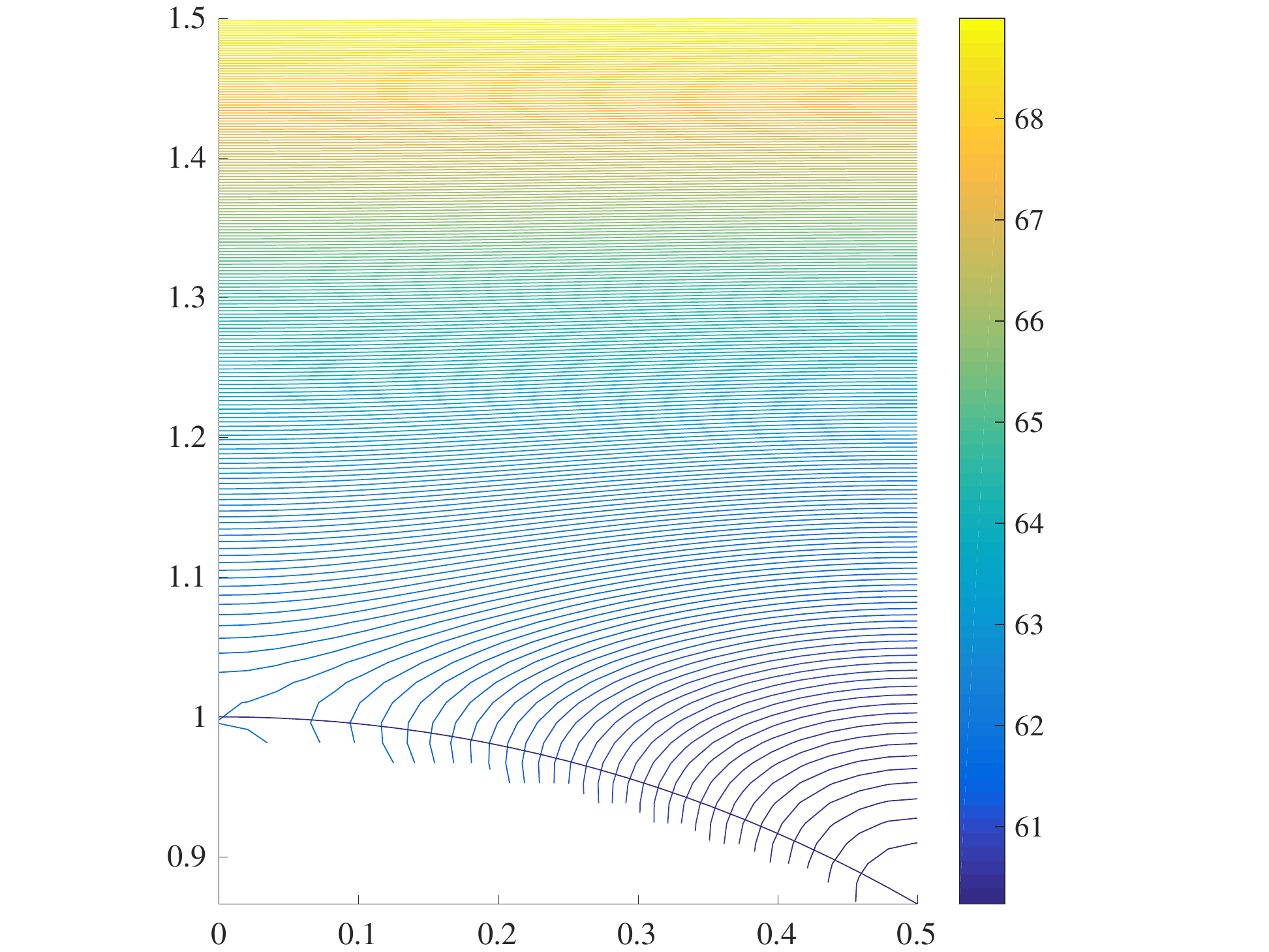}
 \includegraphics[height=.3\textwidth]{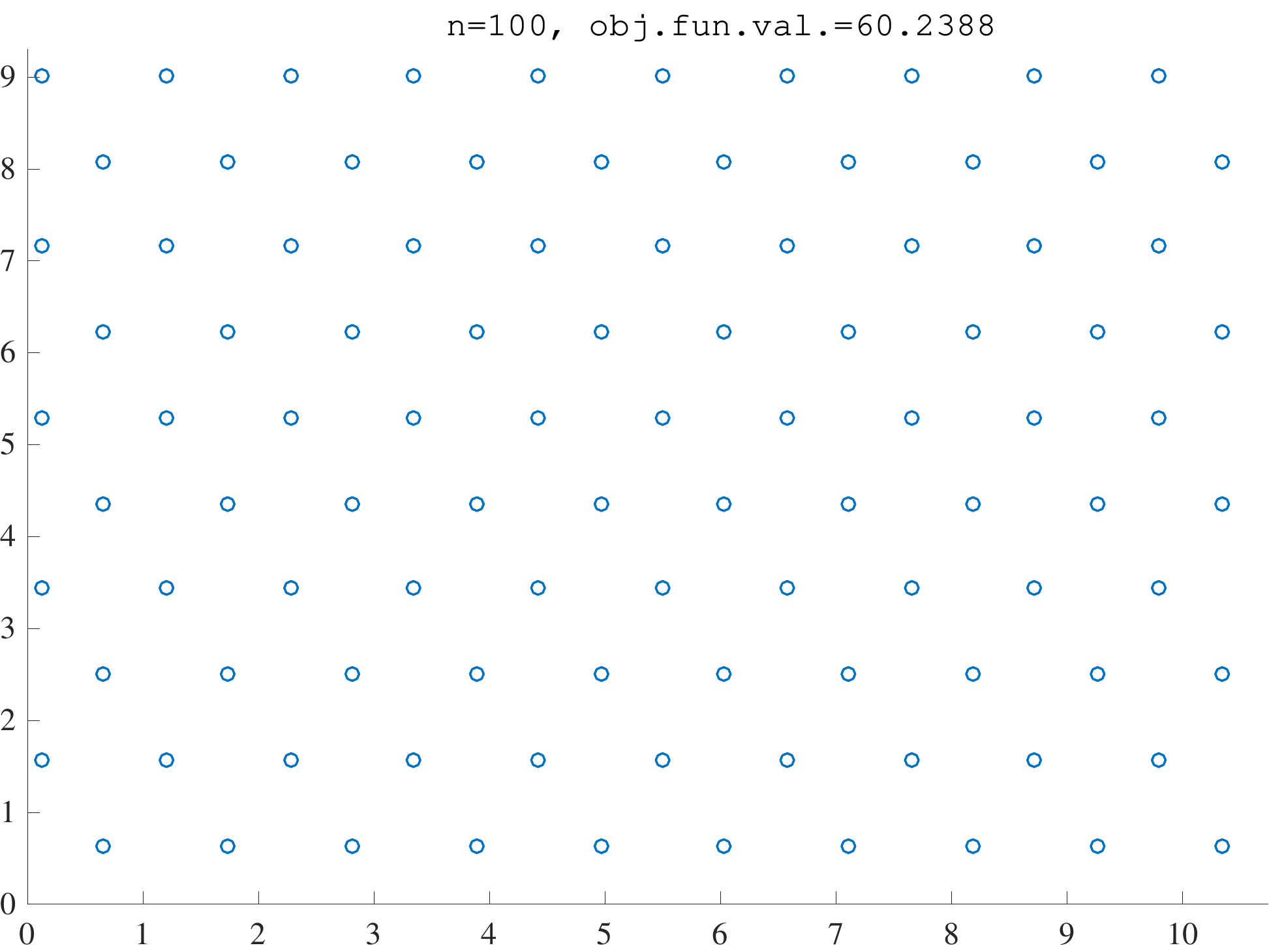}
\caption{ {\bf (left)}  A contour plot of the objective \eqref{eq:traceComp} with  $f(r) = e^{-2 r}$  for configurations $\{ \x_i\}_{i \in [N^2]} =  \Lambda^{10}_{a,b}$ given in \eqref{eq:FiniteToriGraph} as the lattice parameters $(a,b)$ vary over the set $U$ in Proposition~\ref{prop:LatticeParam}. A sector of the unit circle is drawn for reference.
{\bf (right)} The best general pointset configuration obtained for this objective on a flat torus. See Section~\ref{sec:Thompson}. 
\label{fig:traceFundDom}}
\end{center}
\end{figure}

We also consider a general configuration of $n=10^2$ points distributed on a $W\times H$ rectangle chosen such that $W\cdot H = n$ and $H = W\cdot \sqrt{3}/2$  with periodic boundary conditions (a flat torus). We use a BFGS quasi-Newton method with the gradient computed using Proposition~\ref{prop:objFunDiffL} to find a locally optimal configuration. We initialize using a random selection of points chosen independently and uniformly from the rectangle. To avoid local minima, this experiment is repeated several times and  the pointset configuration with smallest objective is plotted in Figure~\ref{fig:traceFundDom}(right). We observe that this pointset configuration closely approximates a triangular lattice. We remark that for the finite pointset case, the triangular lattice configuration is not optimal if the potential is too long-range, even if it is completely monotone, {\it e.g.}, for the above parameters with $f(r) = \exp(- \alpha r)$ with $\alpha = 0.2$ the triangular lattice is not optimal, even among lattices. 

In what follows, we find that several optimal configurations are the triangular lattice. Rather than plotting a figure that appears identical to Figure \ref{fig:traceFundDom}(right) each time, we  instead report the (shifted) value of the closest nearest neighbors in the configuration, 
$$
d_{\text{min}}^\x := \frac{W}{\sqrt{n}} - \min_{i\in [n]} \ \min_{j\neq i} \ d(x_i, x_j) .   
$$
This value is non-negative and a value that is (nearly) zero implies that the configuration is triangular. For the configuration shown in  Figure \ref{fig:traceFundDom}(right), $d_{\text{min}}^\x =   1.34 \times 10^{-5}$. This small discrepancy could be further reduced by, {\it e.g.} reducing the convergence criterion for the optimization method.

\begin{figure}[t]
\begin{center}
\includegraphics[height=.24\textwidth]{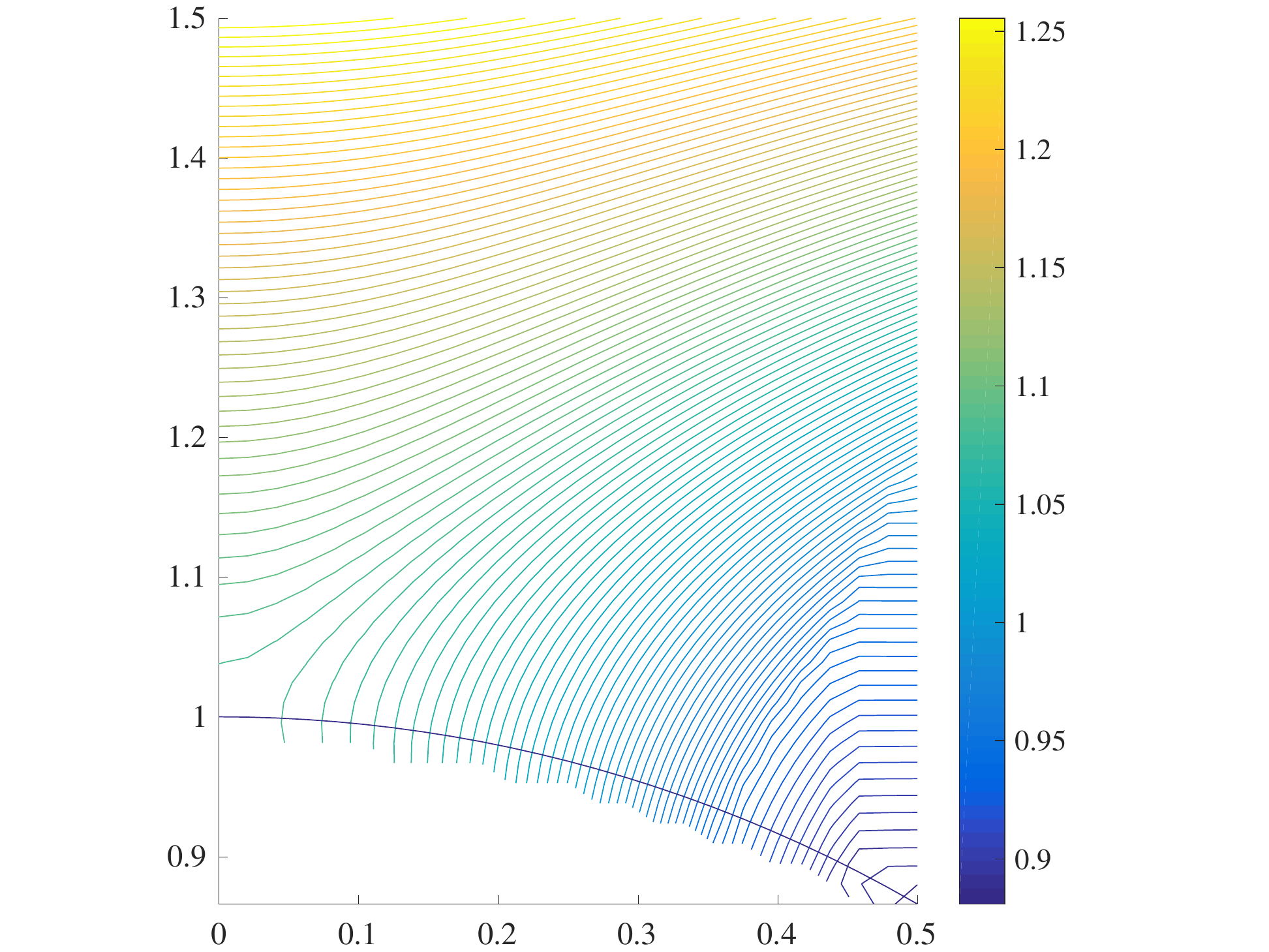}
\includegraphics[height=.24\textwidth]{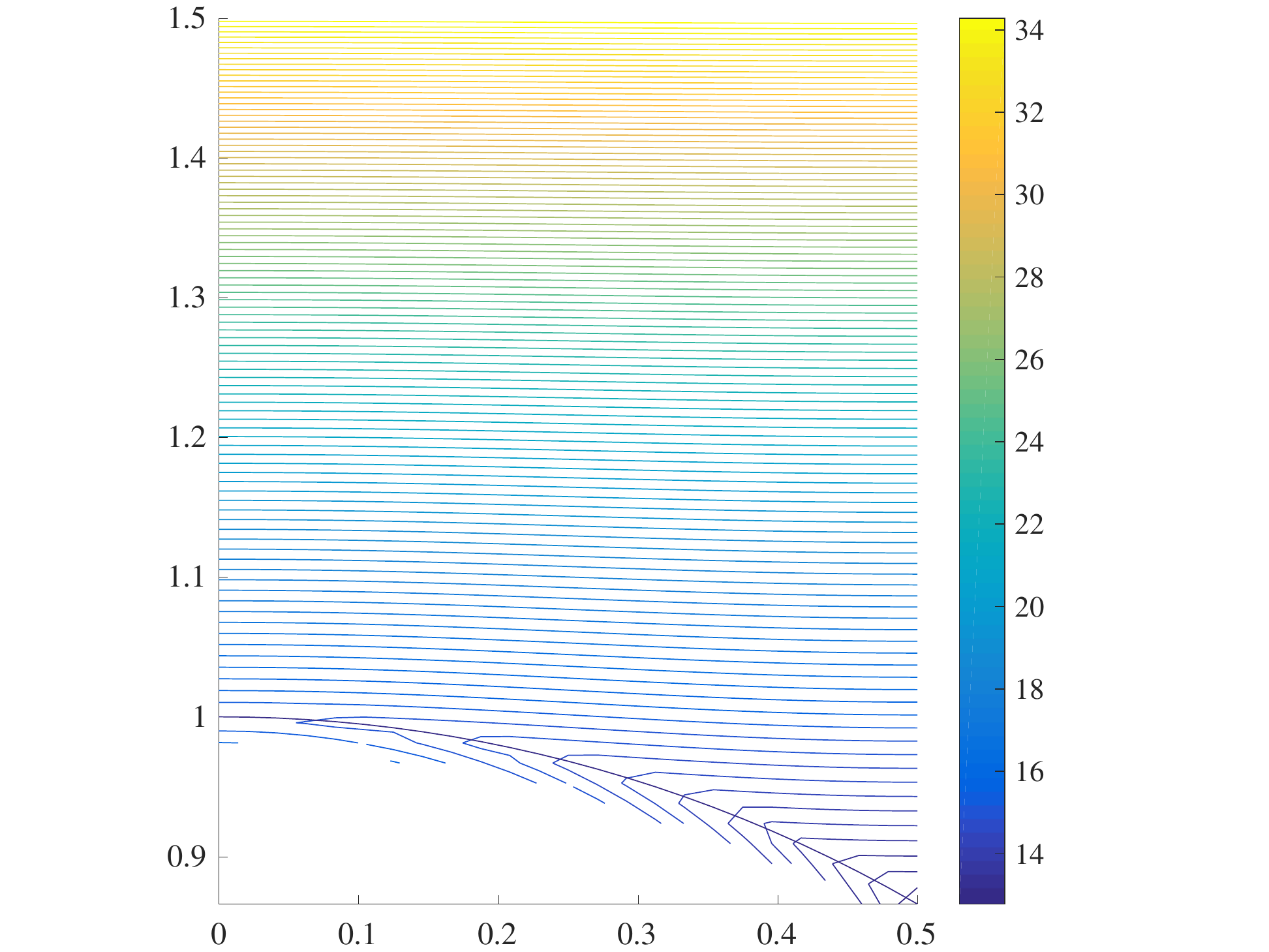}
\includegraphics[height=.24\textwidth]{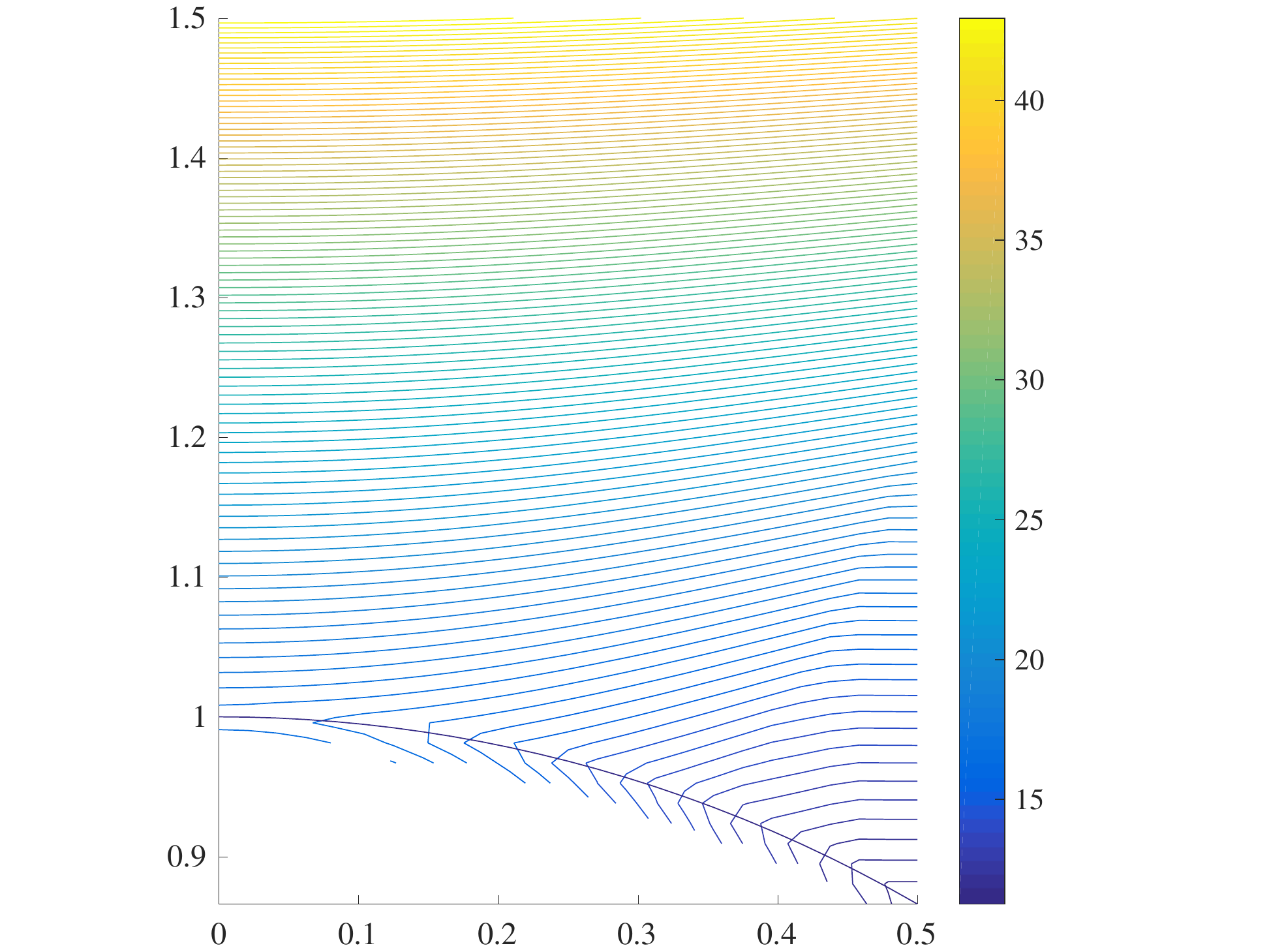}
\caption{Plots of the spectral radius $\lambda^\x_n$ (left), 
inverse algebraic connectivity $1/\lambda^\x_2$ (center), and 
condition number $ \lambda^\x_n/ \lambda^\x_2$ (right)
with  $f(r) = e^{- 2r} $ for configurations $\{ \x_i\}_{i \in [N^2]} = \Lambda^N_{a,b}$ with $(a,b) \in U$.  
See Sections~\ref{sec:SpecRad}, \ref{sec:AlgConn}, and \ref{sec:CondNum}. }
\label{fig:FundDoms}
\end{center}
\end{figure}

\subsection{Spectral radius} \label{sec:SpecRad}
We consider the spectral radius of the graph Laplacian, $\lambda_{\max}(L^\x)$,  as defined in \eqref{eq:LSpecRad}. 
We choose parameters and $f$ as in Section~\ref{sec:Thompson}.  In Figure~\ref{fig:FundDoms}(left), we plot $\lambda_{\max}(L^\x)$ for the configuration $\{ \x_i\}_{i \in [N^2]} = \Lambda^{a,b}_N$ as the lattice parameters $(a,b)$ vary over the set $U$ in Proposition~\ref{prop:LatticeParam}. 
For general pointset configurations, we obtain a configuration that is very close to a triangular lattice, with $d_{\text{min}}^\x = 1.8\times 10^{-4}$. We remark that for this objective, there are many more local minima  where the pointset becomes ``geometrically frustrated'' and is unable to converge to a triangular configuration.

\subsection{Algebraic connectivity} \label{sec:AlgConn}
We first consider the inverse algebraic connectivity of the graph, $1/\lambda_2 \left( L^\x \right)$, as defined in \eqref{eq:lam2}.
We choose parameters and $f$ as in Section~\ref{sec:Thompson}.  
In Figure~\ref{fig:FundDoms}(center), we plot $1/\lambda_2 \left( L^\x \right)$  for the configuration $\{ \x_i\}_{i \in [N^2]} = \Lambda^{a,b}_N$ as the lattice parameters $(a,b)$ vary over the set $U$ in Proposition~\ref{prop:LatticeParam}. For equal-volume  lattice configurations, the triangular lattice is minimal, but for  general pointset configurations, the minimal  configuration obtained is when the points coalesce  to a single point.

\subsection{Condition number} \label{sec:CondNum}
We consider the condition number $\kappa(L^\x) = \lambda^\x_n/ \lambda^\x_2$, as defined in \eqref{eq:CondNum}.  
We choose parameters and $f$ as in Section~\ref{sec:Thompson}.  
In Figure~\ref{fig:FundDoms}(right), we plot $\kappa(L^\x)$  for the configuration $\{ \x_i\}_{i \in [N^2]} = \Lambda^{a,b}_N$ as the lattice parameters $(a,b)$ vary over the set $U$ in Proposition~\ref{prop:LatticeParam}. 
Even though the condition number of a graph is a quasi-convex function of the graph weights, the triangular lattice is minimal among lattices of the same volume. 
For general pointset configurations, we obtain a configuration that is very close to a triangular lattice, with $d_{\text{min}}^\x = 9.16\times10^{-6}$.

\subsection{Total effective resistance} \label{sec:EffRes}
We consider the total effective resistance  of the graph, $R_{\mathrm{tot}}^\x$, as defined in \eqref{eq:EffRes}. 
We choose parameters as in Section~\ref{sec:Thompson}, except we choose the concave and increasing function $f(r) = 1 - \exp(-\alpha r)$ with $\alpha = 2$.  In Figure~\ref{fig:effResFundDom}(left), we plot $R_{\mathrm{tot}}$ for the configuration $\{ \x_i\}_{i \in [N^2]} = \Lambda^{a,b}_N$ as the lattice parameters $(a,b)$ vary over the set $U$ in Proposition~\ref{prop:LatticeParam}. 
For general pointset configurations, we obtain an optimal configuration that is very close to a triangular lattice, with $d_{\text{min}}^\x = 9.9\times 10^{-6}$. 

If we choose the convex and decreasing function $f(r) = \exp(- 2 r)$ instead, 
one observes from Figure~\ref{fig:effResFundDom}(right)  that the triangular lattice is minimal among lattices. 
However, for general pointset configurations, the optimal configuration is when  all points are at the same location. 
For this objective, the intermediate configurations along the optimization path are very interesting. 
Iterations  1, 10, 150, 200, 250, and 400 for a initial condition are plotted in Figure~\ref{fig:effResFundTime}. 
We observe that the points rapidly aggregate along one dimensional curves which becomes hexagonal before coalescing  to a single  point. 
For other random initial conditions, the one dimensional curves formed at intermediate iterations are sometimes geodesics of the torus.

\begin{figure}[t]
\begin{center}
 \includegraphics[height=.23\textwidth]{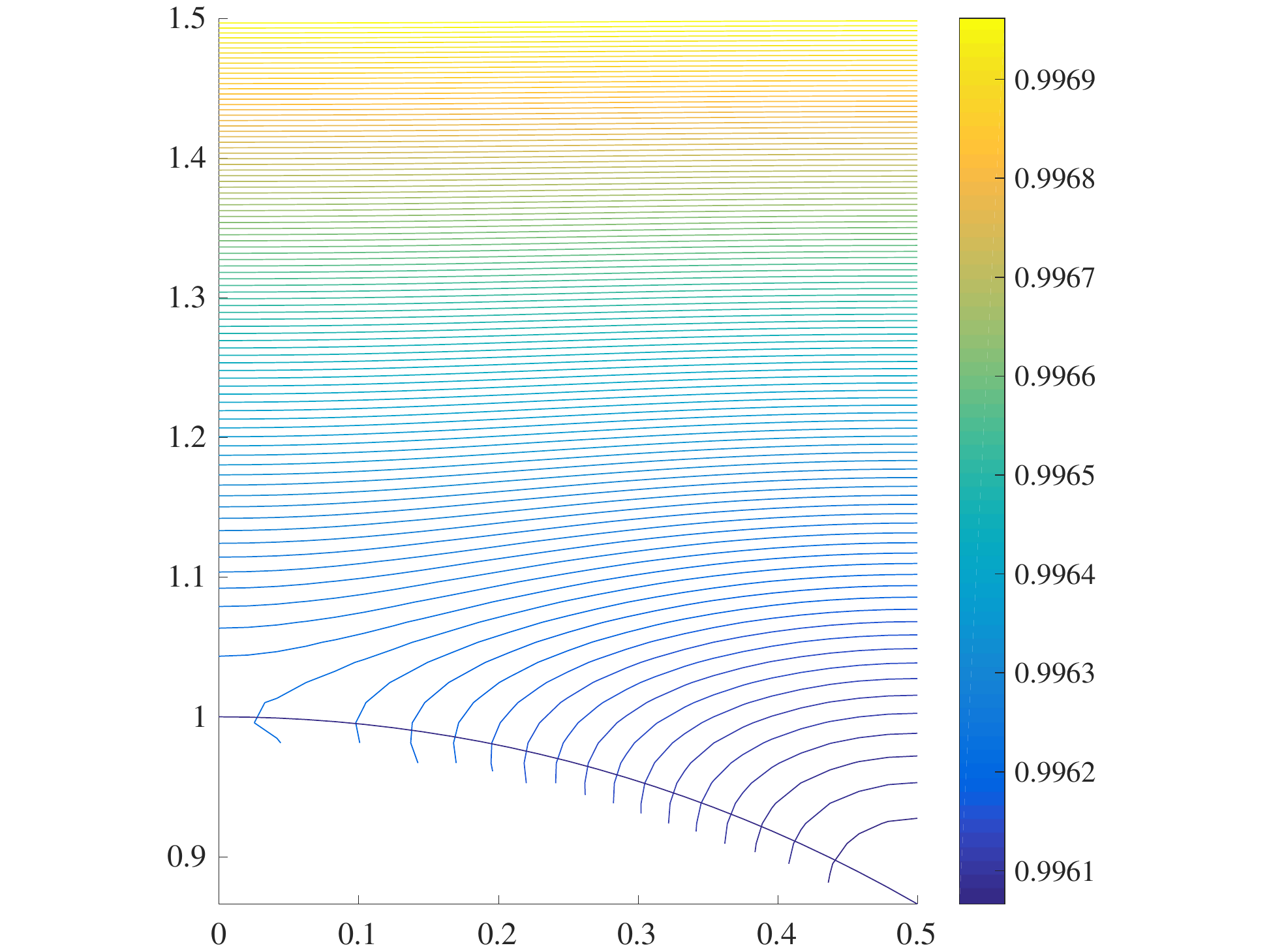}
 \includegraphics[height=.23\textwidth]{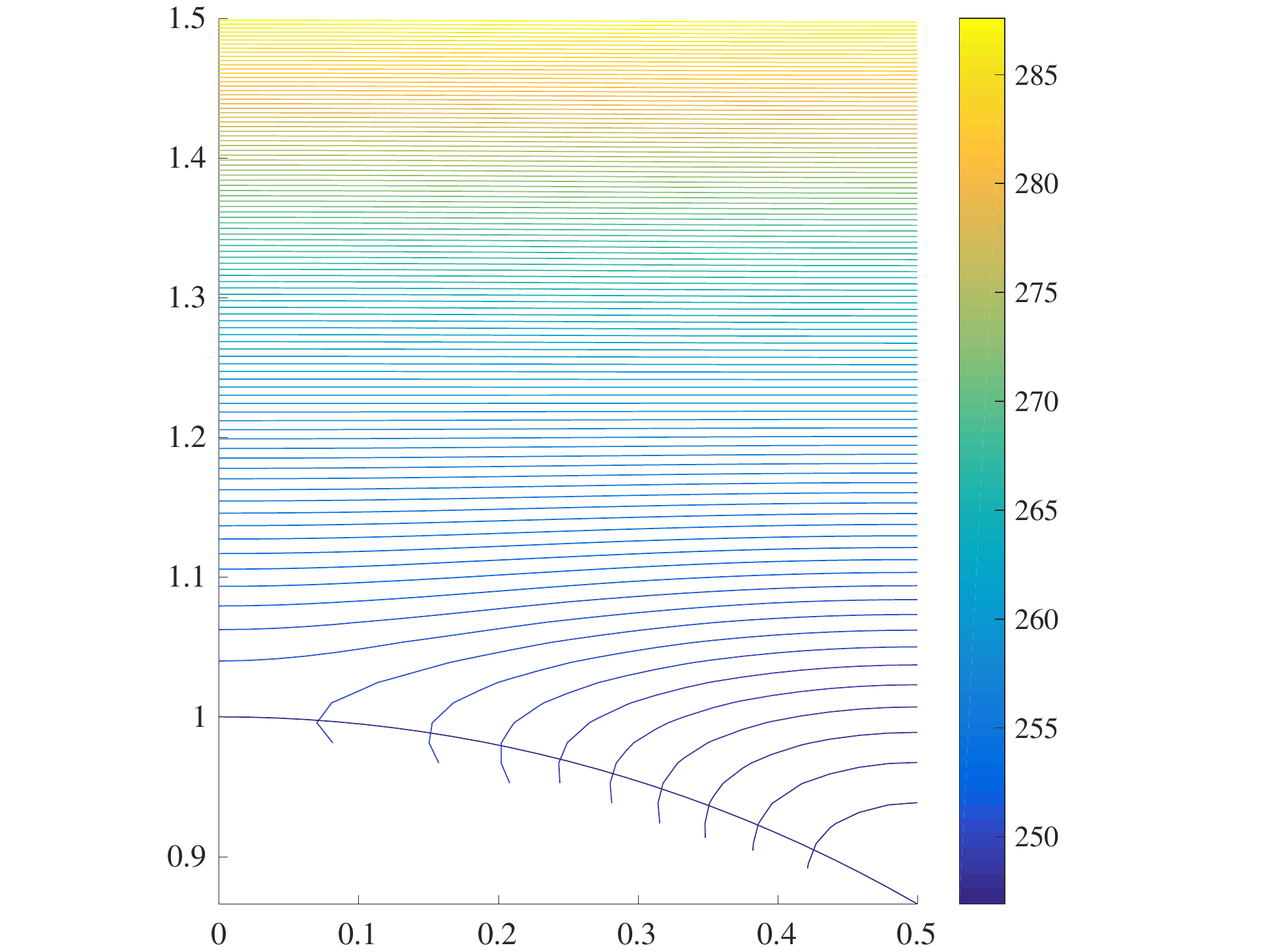}
 \includegraphics[height=.23\textwidth]{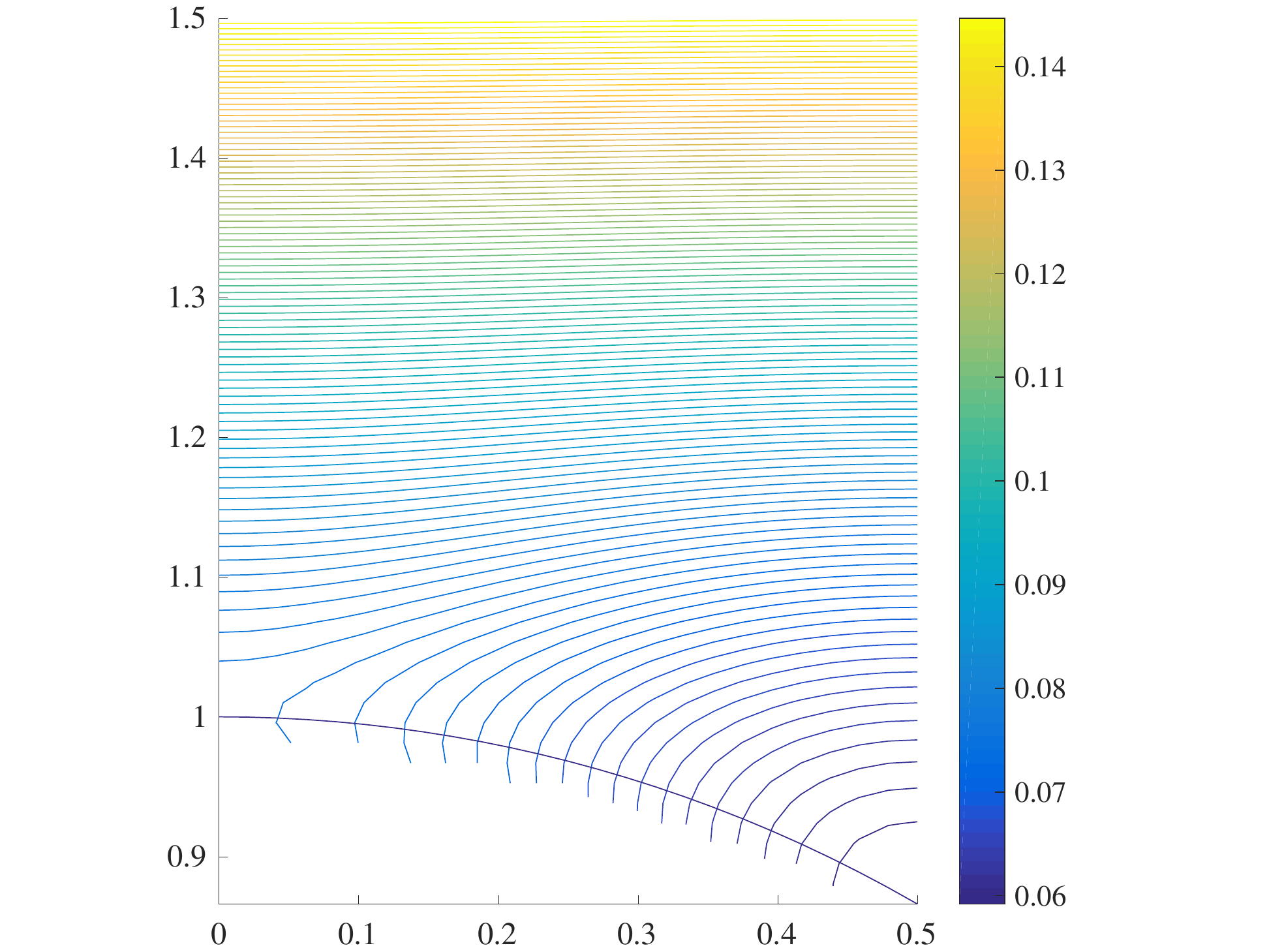}
\caption{ 
{\bf (left)} Plot of the effective resistance with $f(r) = 1 - e^{-2r} $for configurations $\{ \x_i\} = \Lambda_{a,b}^N$ as the lattice parameters $(a,b)$ vary over the set $U$ in Proposition~\ref{prop:LatticeParam}.  
{\bf (center)} The same except $f(r) = e^{- 2r} $. 
{\bf (right)} Same plot, but for the variance of the graph Laplacian eigenvalues. 
See Sections~\ref{sec:EffRes} and \ref{sec:Var}. }
\label{fig:effResFundDom}
\end{center}
\end{figure}

\begin{figure}[t]
\begin{center}
\includegraphics[height=.23\textwidth]{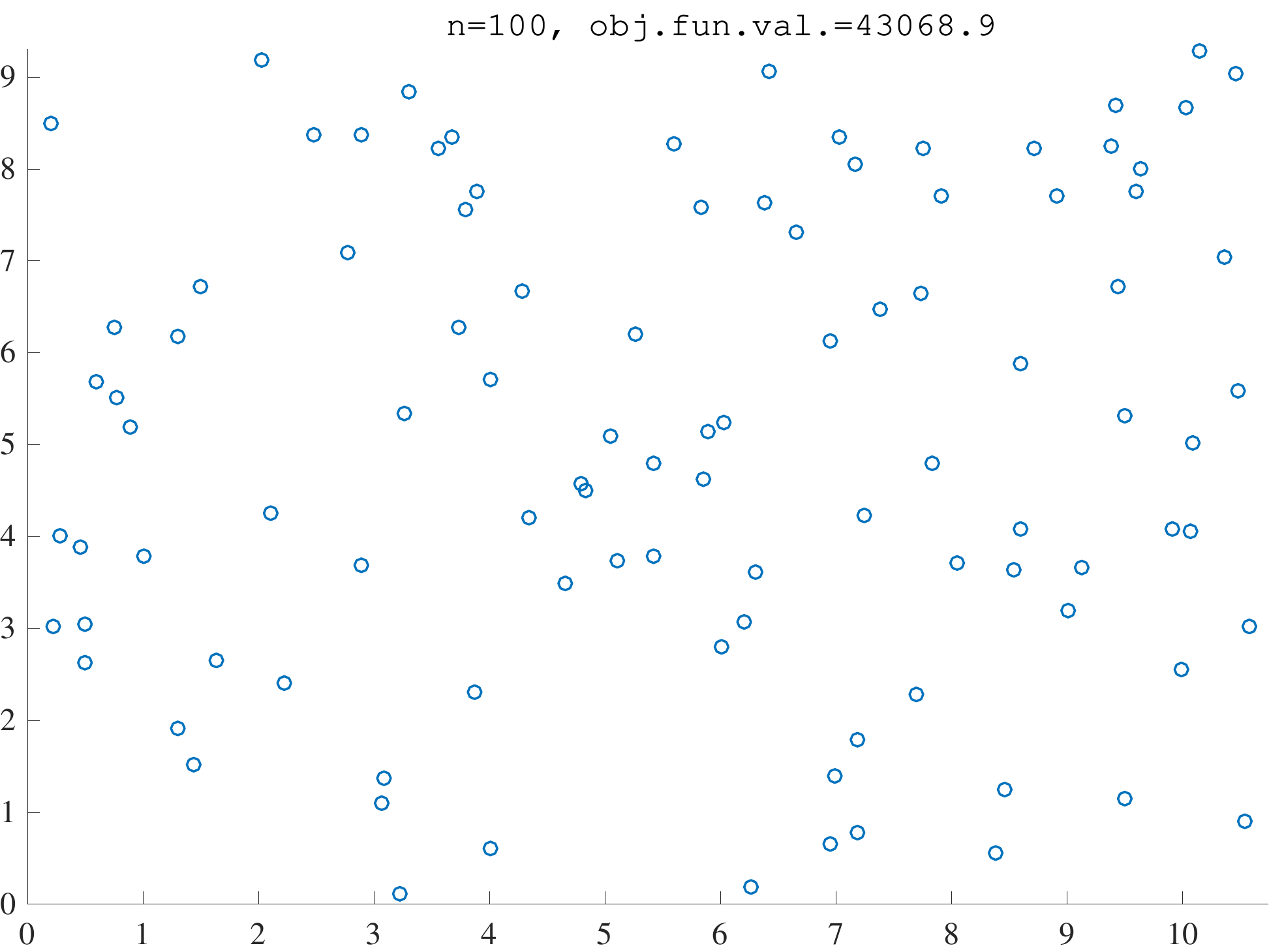}
\includegraphics[height=.23\textwidth]{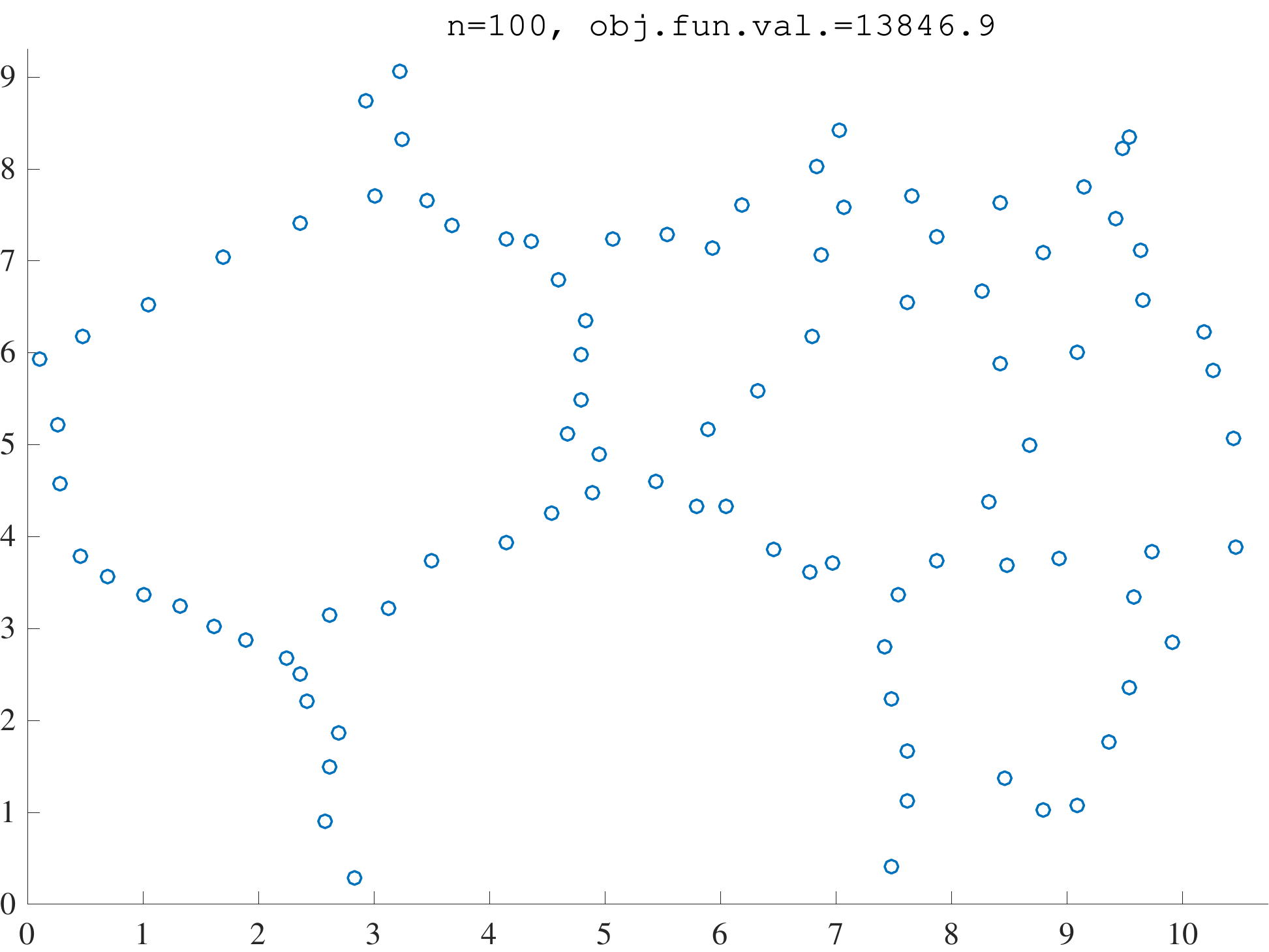}
\includegraphics[height=.23\textwidth]{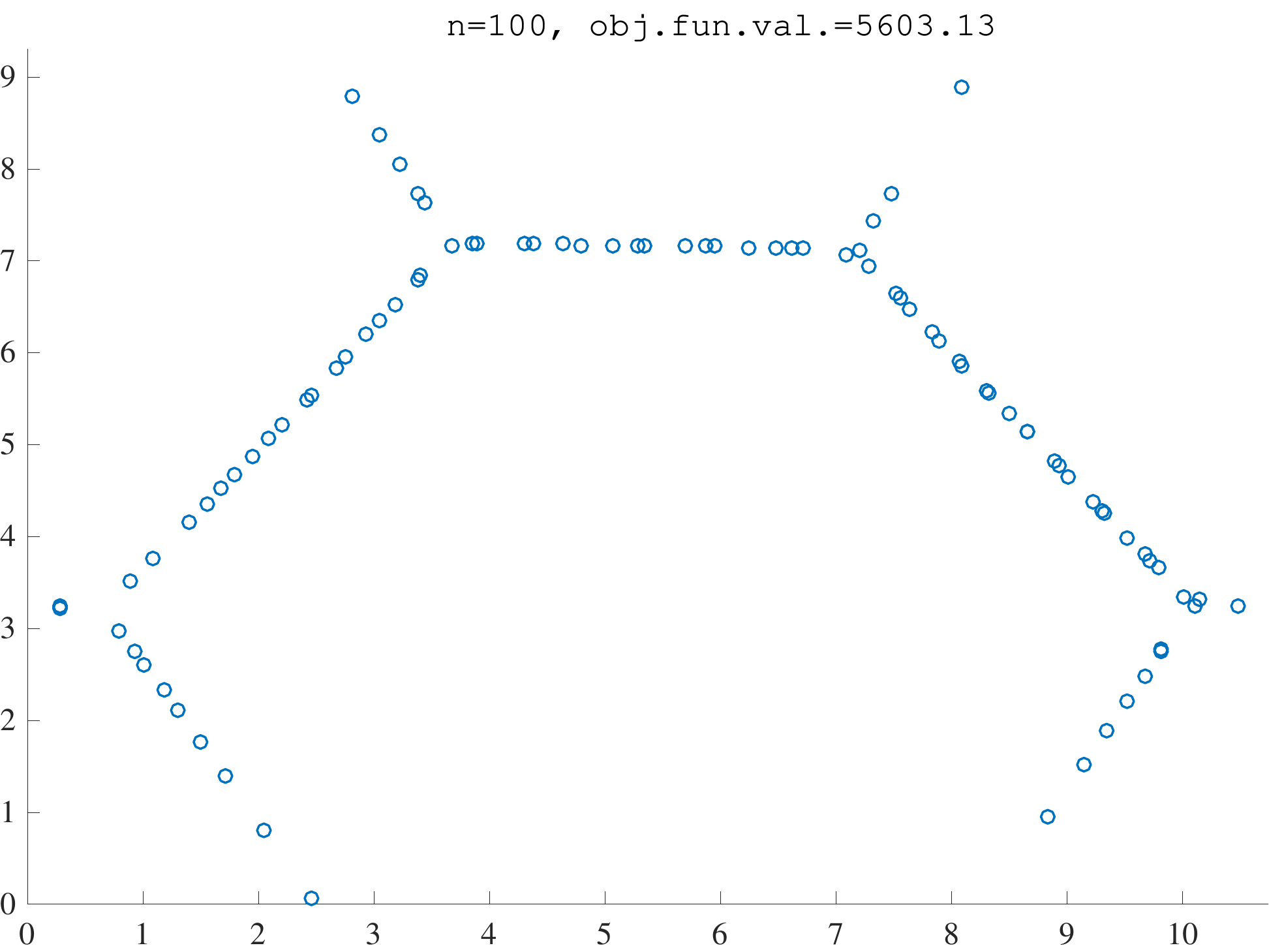}
\includegraphics[height=.23\textwidth]{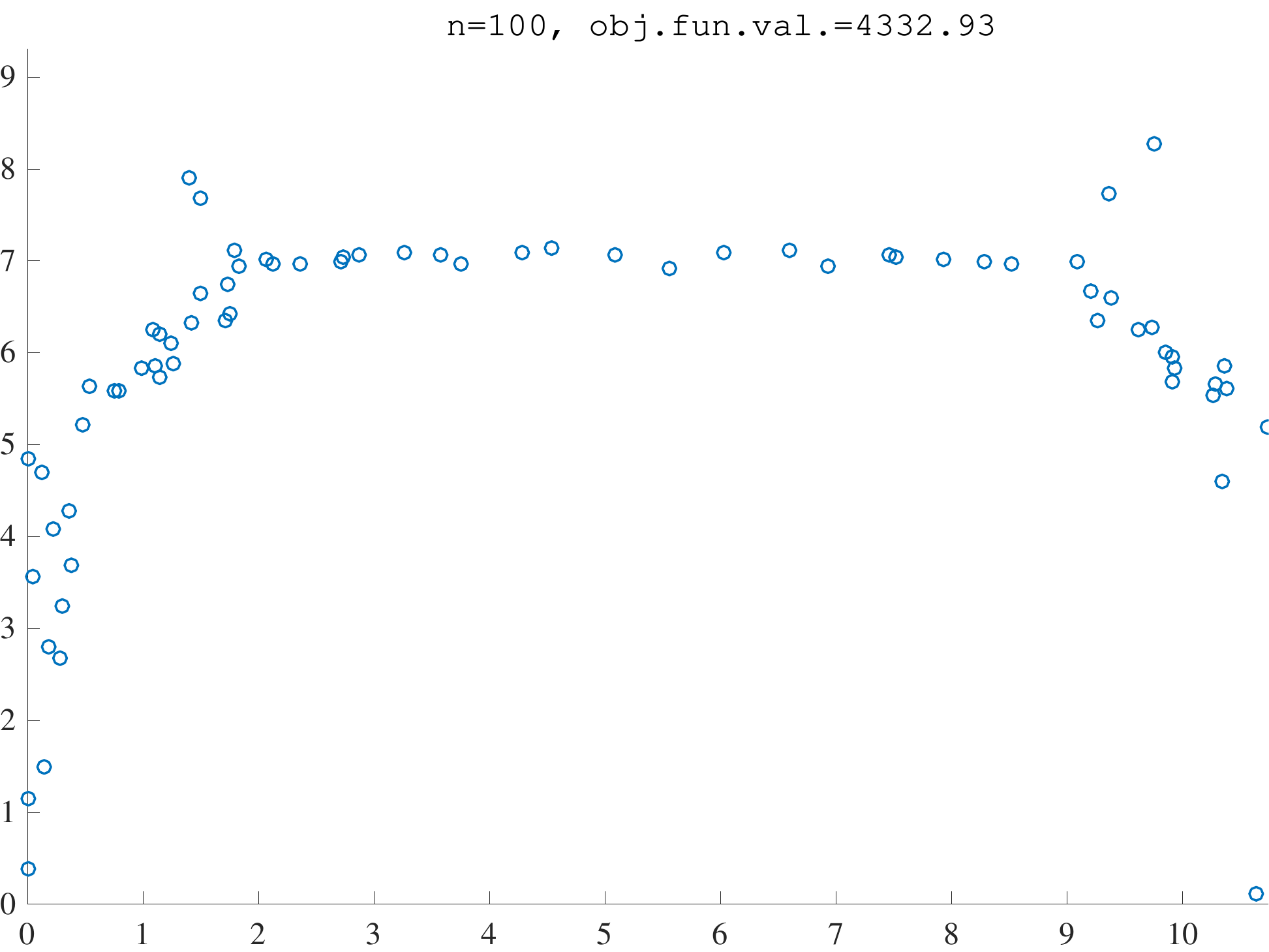}
\includegraphics[height=.23\textwidth]{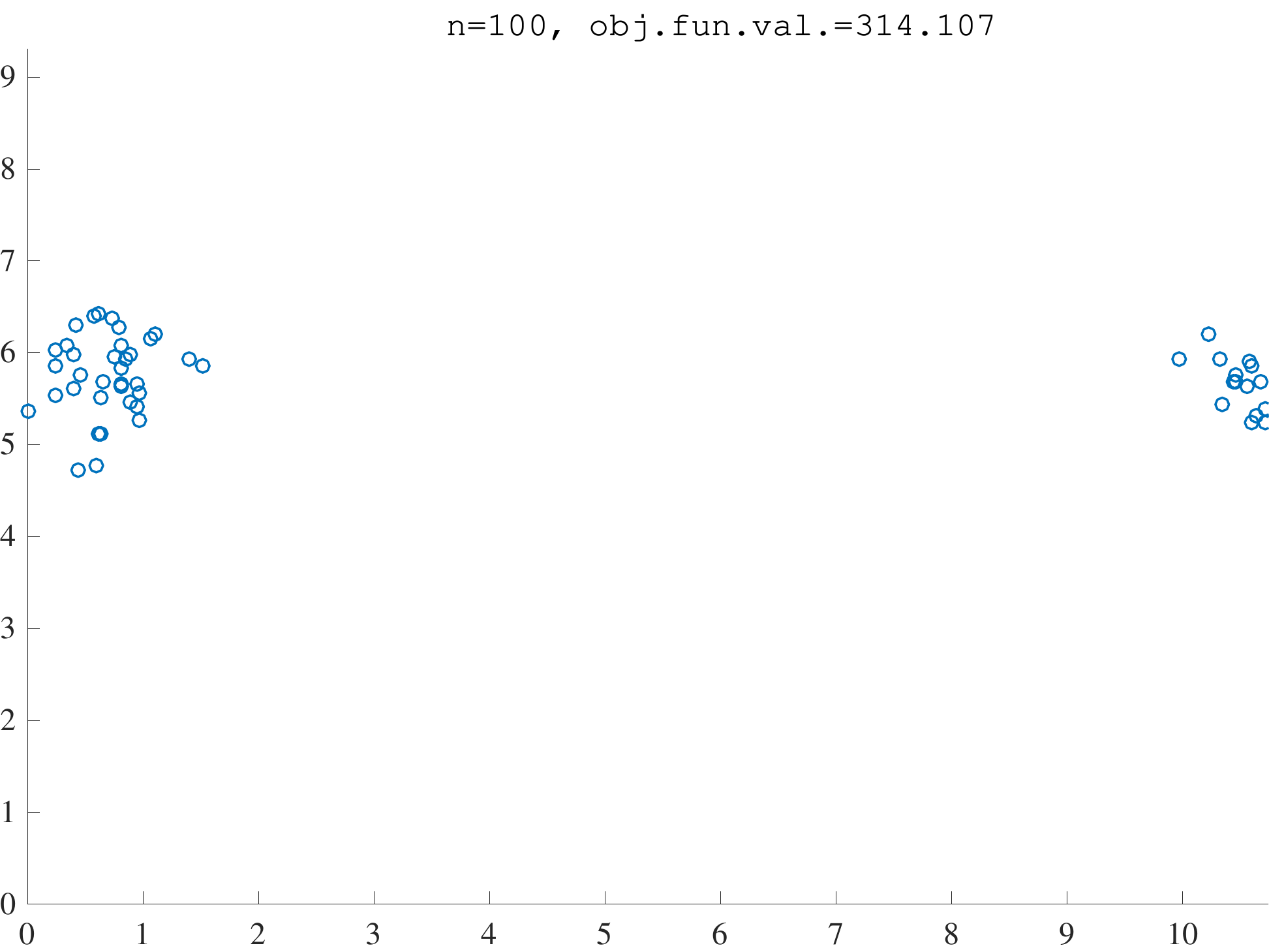}
\includegraphics[height=.23\textwidth]{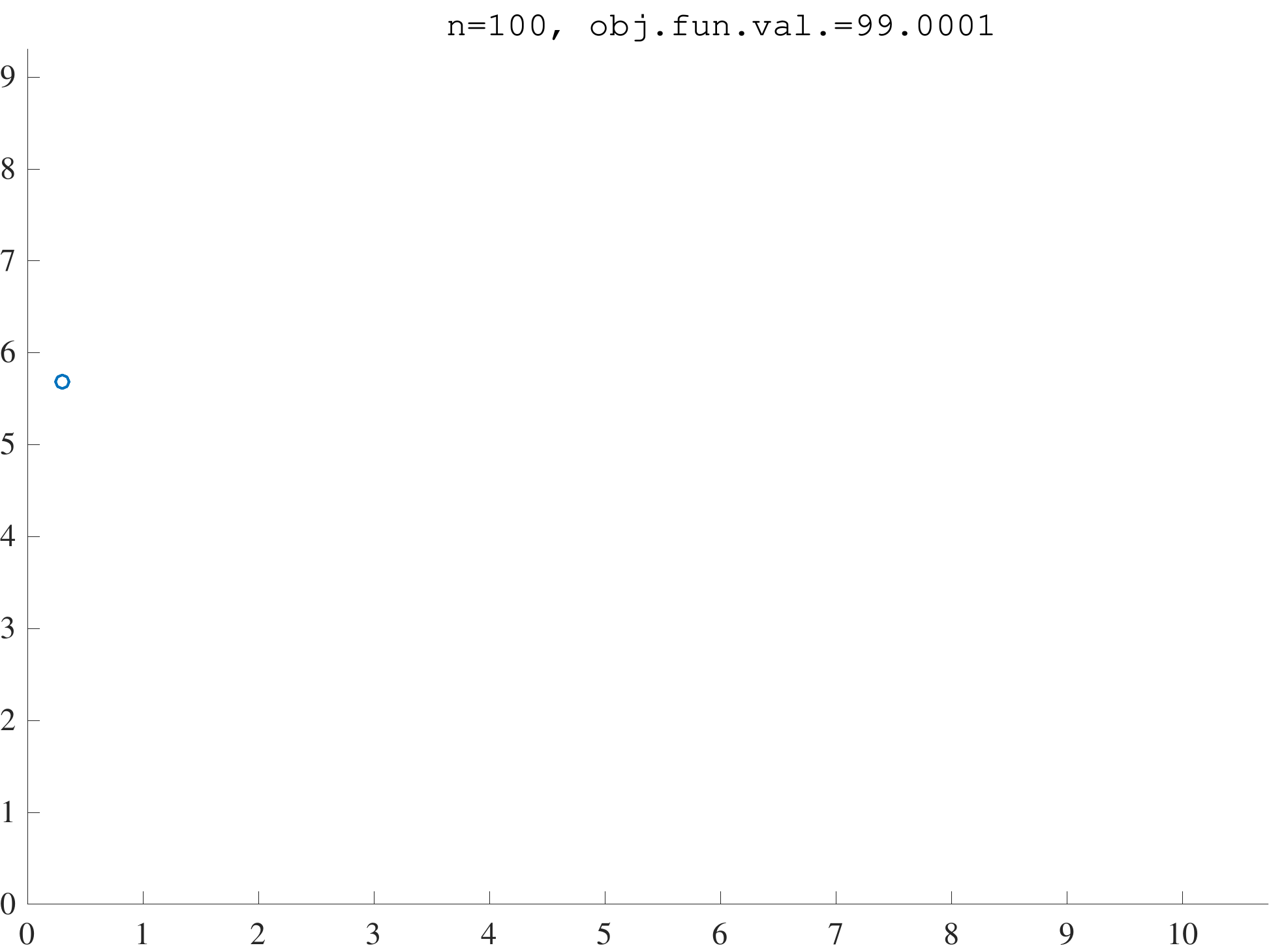}
\caption{Pointset configurations for iterations 1, 10, 150, 200, 250, and 400 for the  problem of minimizing the effective resistance as described in Section \ref{sec:EffRes}. }
\label{fig:effResFundTime}
\end{center}
\end{figure}

\subsection{Variance of the graph Laplaican eigenvalues} \label{sec:Var}
We consider the variance of the graph Laplacian eigenvalues, 
$$
V^\x = \frac{1}{n}\sum_{i\in[n]} (\lambda_i^\x)^2 -  \left(\frac{1}{n} \sum_{i\in[n]} \lambda^\x_i \right)^2. 
$$
We choose parameters and $f$ as in Section~\ref{sec:Thompson}.  In Figure~\ref{fig:effResFundDom}(right), we plot $V^\x$ for the configuration $\{ \x_i\}_{i \in [N^2]} = \Lambda^{a,b}_N$ as the lattice parameters $(a,b)$ vary over the set $U$ in Proposition~\ref{prop:LatticeParam}. 
For general pointset configurations, we again obtain a configuration that is very close to a triangular lattice, with $d_{\text{min}}^\x = 1.1\times 10^{-6}$. 

\subsection{Distance to an interval} \label{sec:interval}
Finally, we report on an objective function that is inspired by the solar cell design problem. 
Recall that in this application, it is advantageous to have many eigenvalues near a particular value (so that the corresponding states can couple to the solar radiation). 
We consider an interval, say $[\sigma_-, \sigma_+]$,  where the solar spectrum is concentrated. We consider the following objective, 
$$
\Sigma^\x = \sum_{j\in [n]}  \   |\lambda^\x_j - \sigma_-| +  |\lambda^\x_j - \sigma_+| - |\sigma_+ - \sigma_-| ,
$$ 
which measures the sum of the distances between the eigenvalues and the interval, $[\sigma_-, \sigma_+]$. If all of the eigenvalues were contained in the interval $[\sigma_-, \sigma_+]$, then the objective function value would be zero. 
We chose a variety of values for  $\sigma_\pm$, where the center of the interval is nearly  the mean of the eigenvalues for the triangular lattice and the width of the interval is proportional to  the standard deviation of the eigenvalues  for the triangular lattice. 
We choose parameters and $f$ as in Section~\ref{sec:Thompson}. For the triangular lattice, the mean of the eigenvalues is $ 0.602$ and the standard deviation is $
3.47 \times 10 ^{-2}$. For choices of $\sigma_\pm$ such that $(\sigma_+ + \sigma_-)/2 \approx  0.602$,   the triangular lattice is the best configuration obtained, both among lattices and general pointset configurations. 
For truncated lattice configurations,   $\{ \x_i\}_{i \in [N^2]} = \Lambda^{a,b}_N$, are qualitatively the same as Figure \ref{fig:traceFundDom}(left). 
In such a setting, we obtain configurations with  $d_{\text{min}}^\x  \approx 1.0\times 10^{-7}$. 

If we shift the interval away from the mean, say   $(\sigma_+ + \sigma_-)/2 \approx  0.85$ with $(\sigma_+ - \sigma_-) = 0.06$, we observe very different behavior. For truncated lattice configurations,   $\{ \x_i\}_{i \in [N^2]} = \Lambda^{a,b}_N$, the optimal configuration is still triangular as show in Figure \ref{fig:distInterval}(left). However, for general pointset configurations, the triangular configuration is no longer even a local minimum. If we initialize with the triangular lattice, the method converges to a local minimum consisting of  ``bunched stripes'' with objective function value $\Sigma^\x = 41.97$, as plotted in Figure \ref{fig:distInterval}(center).  The pointset configuration found with the smallest objective function value is plotted in Figure \ref{fig:distInterval}(right). The objective function value for this configuration is $\Sigma^\x = 41.40$, while the objective function value for the truncated triangular lattice is $\Sigma^\x = 44.74$. This configuration is locally triangular but has several holes in the structure. We observe similar phenomena for intervals that are shifted farther right of the mean. 

\begin{figure}[t]
\begin{center}
\includegraphics[height=.23\textwidth]{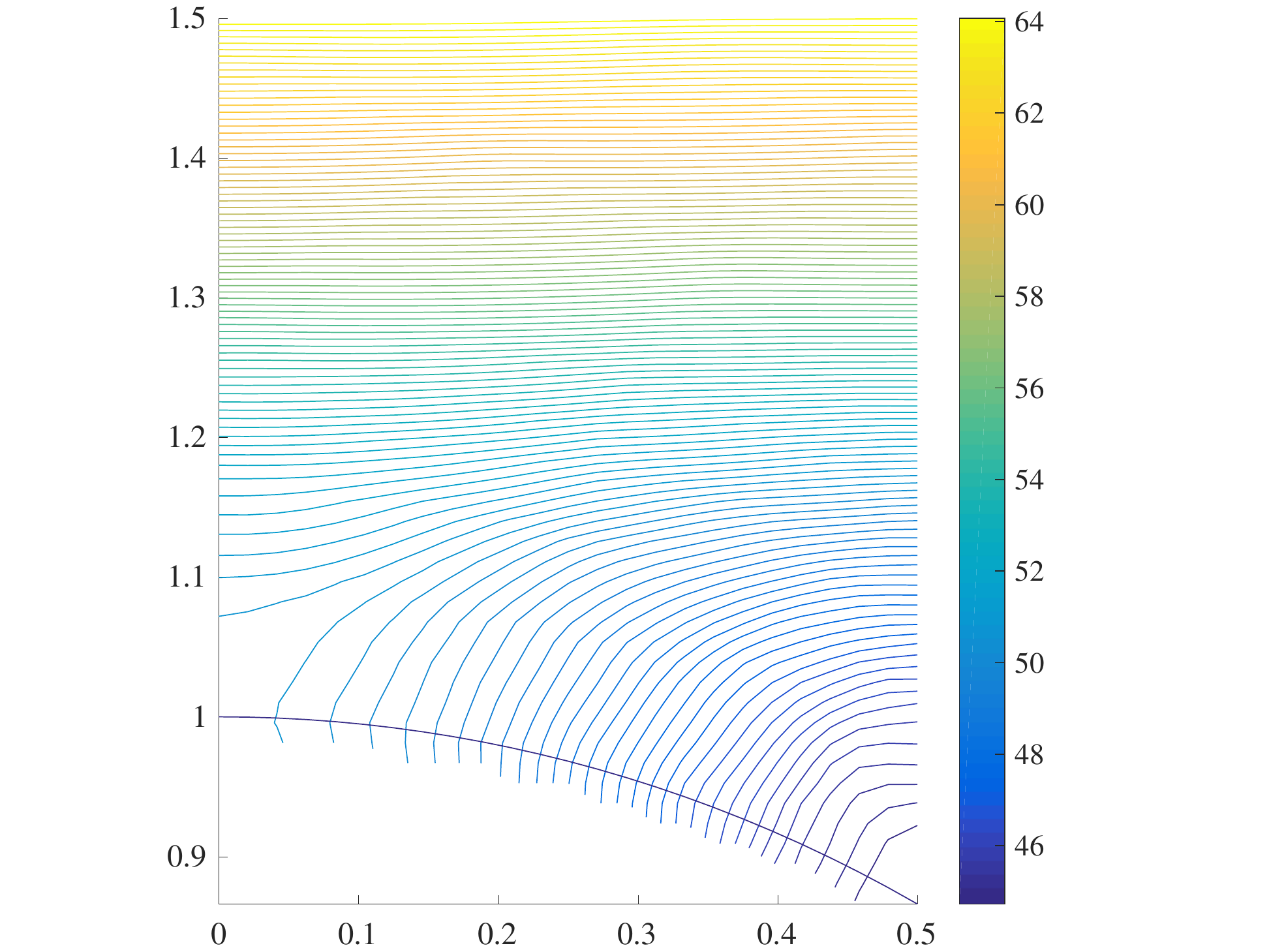}
 \includegraphics[height=.23\textwidth]{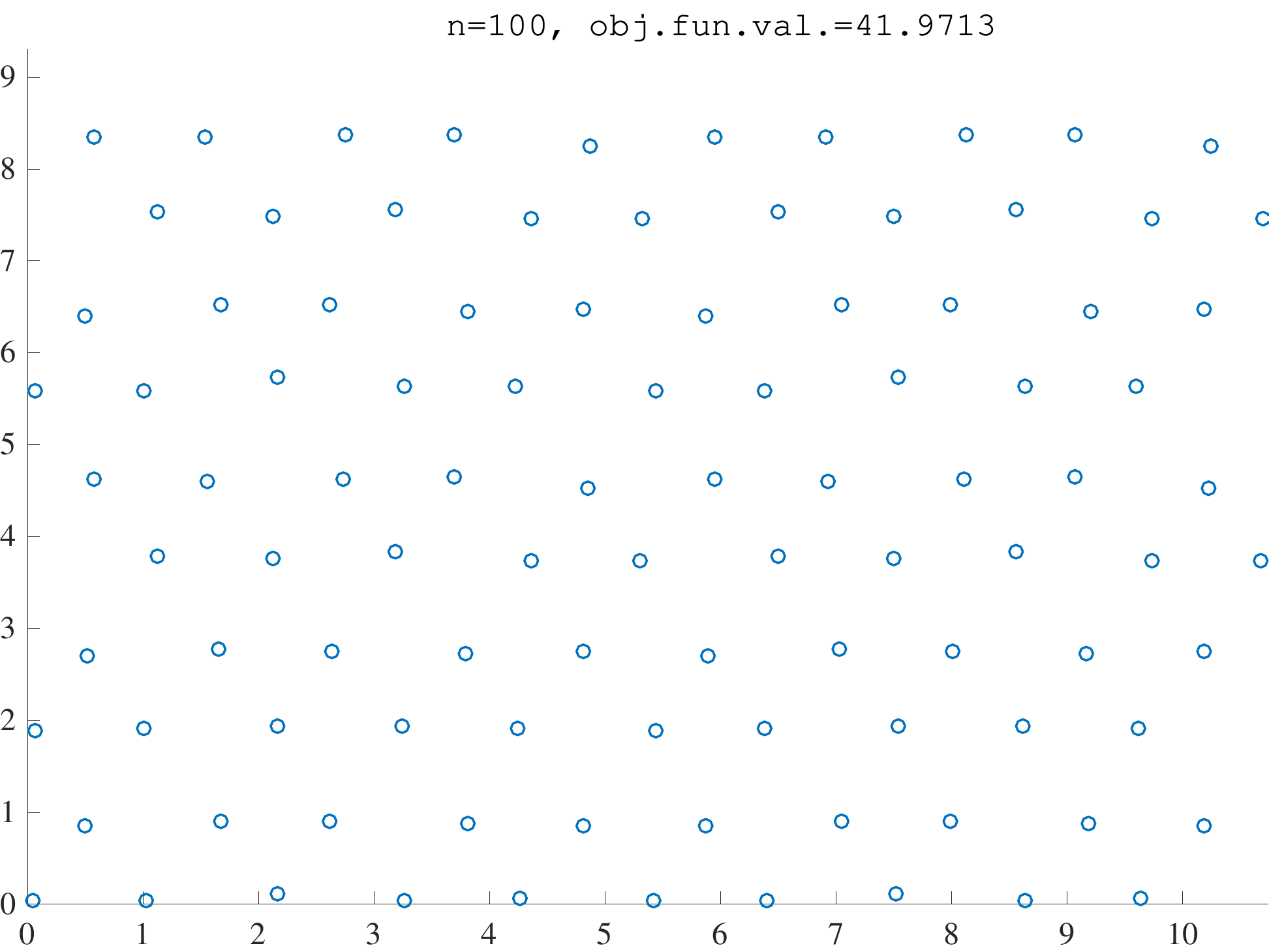}
 \includegraphics[height=.23\textwidth]{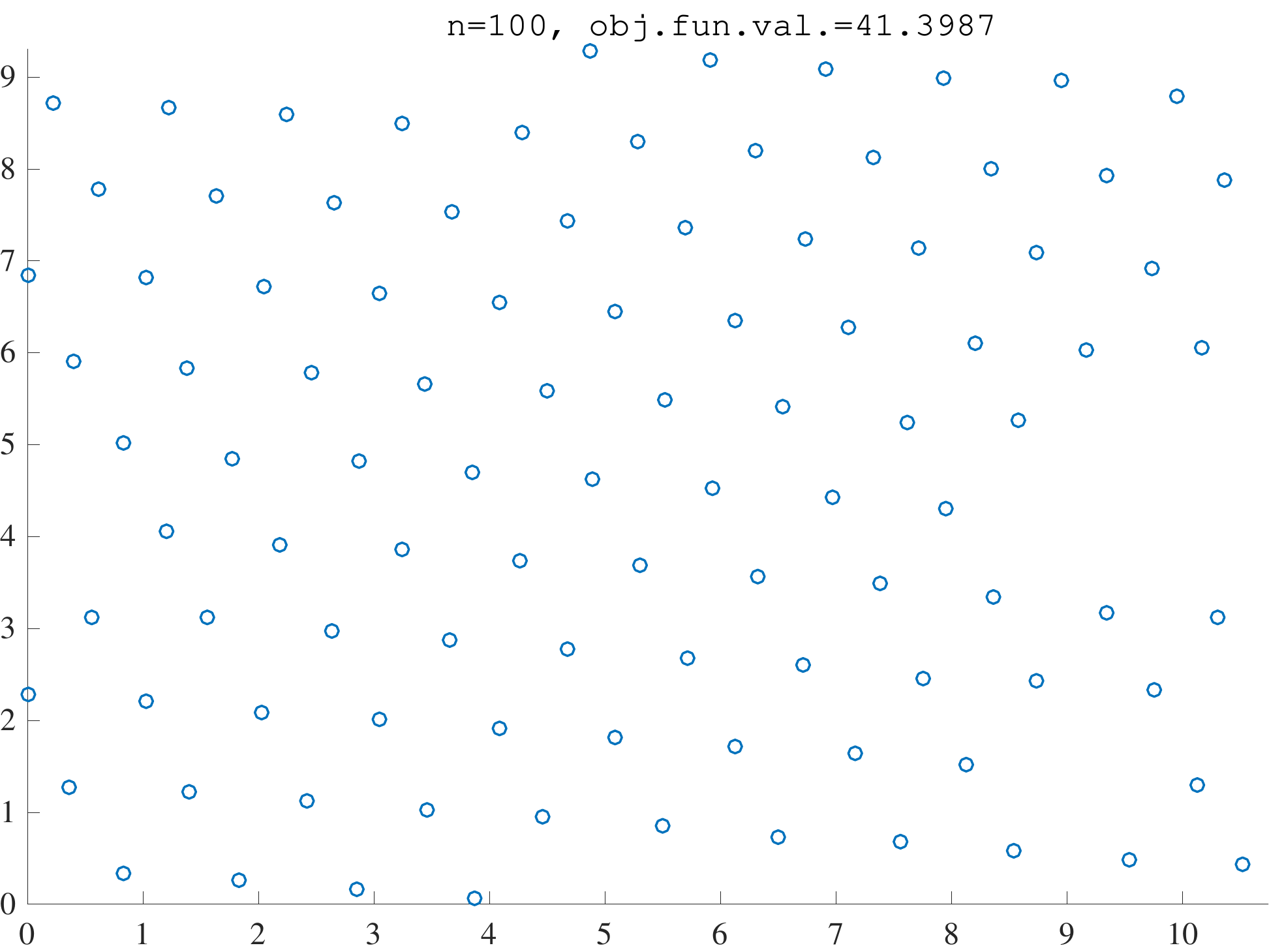}
\caption{ For the objective function in Section \ref{sec:interval}, with a particular choice of $\sigma_\pm$, the triangular lattice is optimal among lattices (left), but not even a local minimum among general pointset configurations (center). The best configuration obtained is plotted in the right panel.   }
\label{fig:distInterval}
\end{center}
\end{figure}

\section{Discussion } \label{sec:disc}
In this paper, we considered spectrally optimal pointset configurations on spheres and tori, as well as spectrally optimal lattices configurations.  On spheres, we have shown that the regular simplex is an optimizer of several spectral quantities using convex analysis.  On lattices, we are able to connect various spectral optimization problems to the natural notions of sphere packings and observe that the triangular lattice is also a ubiquitous optimizer for various spectral quantities.  The key tools involved use notions of spectral measure and Fourier analysis on regular Bravais lattices.  On tori, our results are much weaker (and largely numerical), but we can observe using convergence of spectral measures that equilateral structures appears to still arise naturally in certain limits.  However, the triangular lattice is not optimal among general pointset configurations for all objectives, even for objectives for which the triangular lattice is optimal among truncated lattice configurations. In particular, one interesting setting, discussed in Section \ref{sec:interval}, is minimizing the distance of the spectrum of the graph Laplacian to a fixed interval for a general pointset on a torus.  In this case, for intervals located somewhat far from the mean of that for the truncated triangular lattice, it is possible to observe configurations with smaller objective function than the truncated  triangular lattice.  Characterizing spectral objectives for which the optimal configuration is given by a truncated lattice and identifying spectral objectives which yield non-lattice optimal configurations are very  interesting future directions.

\appendix

 \section{Sensitivity analysis of spectral quantities}  \label{sec:sensitivity analysis}
 In Sections \ref{eq:eigW} and \ref{eq:eigGraphLap}, we introduced the weighted adjacency matrix and graph Laplacian associated to a pointset configuration, discussed some spectral properties of the matrices, and recalled  a variety of spectral quantities that are of interest in various applications. In this appendix, we discuss the sensitivity of these spectral quantities  with respect to changes
 in the pointset configuration.

We say that a function $J \colon \mathbb R^n \to \mathbb R$ is  a \emph{symmetric}  if $J(x) = J([x])$, where $[\cdot]\colon \mathbb R^n \to \mathbb R^n$ rearranges the components of a vector in non-decreasing order. In other words, the value of $J$ is invariant to permuting the components of the argument. Recall that we have defined ${\bf S}^n$ to be the set of real symmetric $n \times n$ matrices and let $\lambda\colon {\bf S}^n \to \mathbb R^n$ be the vector containing the ordered eigenvalues of a symmetric matrix.  We say that the composition $J\circ \lambda$ is a \emph{spectral function} if $J$ is a symmetric function. Roughly speaking, $J\circ \lambda$ is convex and differentiable when $J$ is convex and differentiable \cite{BoLe2006}. 
 
 Recall that the weighted adjacency matrix associated with a Euclidean pointset  $\{\x_i \}_{i\in [n]}$ is defined as in \eqref{eq:W}. 
For the spectral function $J\circ \lambda$, we consider the composition 
  $$ J \circ \lambda \circ W^\x .  $$
 The following proposition shows how the composite function, $  J \circ \lambda \circ W^\x$, changes as we move a single point in the configuration, say $\x_i \in \mathbb R^d$.

 \begin{prop} \label{prop:objFunDiffW} 
 Let $J$ be differentiable at $\lambda(W^\x)$. The gradient of the objective function, $ J \circ \lambda \circ W^\x \colon \mathbb R^{N\times 2} \to \mathbb R$, with respect to $\x_i$ is given by 
$$
\nabla_{\x_i} (J \circ \lambda \circ W^\x) = 4 \sum_k \left( U \  \mathrm{diag} \nabla J (\lambda)  \  U^t \right)_{ik}  
 f'( d^2(\x_i,  \x_k) )  \ (\x_ i - \x_k)
 $$
for any diagonalizing matrix $U\in  O^n$ satisfying $W^\x = U  \  \mathrm{diag}\ \lambda (W^\x) \   U^t$. 
\end{prop}
\begin{proof} The proof is an exercise in the chain rule. 
When $ J$ is differentiable at $\lambda(A)$,  the composition $J\circ \lambda  \colon {\bf S}^N \to \mathbb R$ is Fr\'echet differentiable with derivative 
$$
\nabla (J\circ \lambda) (A) = U \  \mathrm{diag} \nabla J (\lambda ) \  U^t
$$
for any diagonalizing matrix $U\in  O^n$ satisfying $A = U  \  \mathrm{diag}\ \lambda (A) \   U^t$. See, for example, \cite[Corollary 5.2.5]{BoLe2006}. For simplicity, denote  $V = U \  \mathrm{diag} \nabla J (\lambda) \  U^t$. Thus, we compute 
\begin{subequations}
\label{eq:dJ}
\begin{align}
\label{eq:dJa}
\nabla_{\x_i}( J \circ \lambda \circ W^\x) &= \langle V , \nabla_{\x_i} W^\x \rangle_F \\
\label{eq:dJb}
&=  \sum_{jk} V_{jk} \nabla_{\x_i} W_{jk}^\x \\ 
\label{eq:dJc}
&=  \sum_{jk} V_{jk} \left(  \delta_{ij} \nabla_{\x_i}  W_{ik}^\x + \delta_{ik} \nabla_{\x_i}  W_{ji}^\x \right)\\
\label{eq:dJd}
&=  \sum_{k} V_{ik} \   \nabla_{\x_i}  W_{ik}^\x  +  \sum_{j} V_{ji} \  \nabla_{\x_i}  W_{ji}^\x  \\
\label{eq:dJe}
&=  2 \sum_{k} V_{ik} \    \nabla_{\x_i}  W_{ik}^\x 
\end{align}
\end{subequations}
In \eqref{eq:dJa}, $\langle \cdot , \cdot \rangle_F $ denotes the Frobenius inner product. In \eqref{eq:dJc},  $\delta_{ij}$ denotes the Kronecker delta function. In \eqref{eq:dJe}, we used the symmetry of $V$ and $W$.  We compute 
\begin{equation} \label{eq:dwdx}
 \nabla_{\x_i}  W_{ik}^\x = f'( d^2(\x_i,  \x_k) )  \ \nabla _{\x_i}  d^2(\x_i,  \x_k) = 
 2 f'( d^2(\x_i,  \x_k) )  \ (\x_ i - \x_k).
 \end{equation}
 from which the result follows. 
\end{proof}

\bigskip

\noindent {\bf Example.} Consider the  spectral function $J(\lambda) = \sum_i \lambda_i^2$. In this case,  $U \  \mathrm{diag} \nabla J (\lambda)  \  U^t = 2 W^\x$. Thus from Proposition~\ref{prop:objFunDiffW} we can check that 
$$
\nabla_{\x_i} (J \circ \lambda \circ W^\x)  = 2 \langle W^\x , \nabla_{\x_i} W^\x \rangle_F  = \nabla_{\x_i} \| W^\x \|_F^2 .
$$
This is a trivial identity since  $ \| W^\x \|_F^2 =  \mathrm{tr} [ (W^\x)^2 ] = \sum_i \lambda_i(W^\x)^2$. 

\bigskip

As in \eqref{eq:L}, the graph Laplacian associated with a pointset $\{\x_i \}_{i\in [n]}$ is given by $L^\x = D^\x - W^\x$.  
For the spectral function $J\circ \lambda$, we consider the composition 
  $$
  J \circ \lambda \circ L^\x. 
  $$
The following proposition shows how this composite function, $    J \circ \lambda \circ L^\x$, changes as we move a single point in the configuration, say $\x_i \in \mathbb R^n$. 

\begin{prop} \label{prop:objFunDiffL} 
Let $J$ be differentiable at $\lambda(L^\x)$.  The gradient of the objective function, $ J \circ \lambda \circ L^\x \colon \mathbb R^{N\times 2} \to \mathbb R$, with respect to $\x_i$ is given by 
\begin{equation}
\label{eq:dJLap}
\nabla_{\x_i} (J \circ \lambda \circ L^\x) = 2 \sum_k(V_{kk} - 2 V_{ik} + V_{ii})  \
 f'( d^2(\x_i,  \x_k) )  \ (\x_ i - \x_k)
 \end{equation}
where $V =  \left( U \  \mathrm{diag} \nabla J (\lambda)  \  U^t \right)_{ik}$ for any diagonalizing matrix $U\in  O^n$ satisfying $L^\x = U  \  \mathrm{diag}\ \lambda (L^\x) \   U^t$. The gradient can alternatively  be expressed using the arc-vertex incidence matrix decomposition  \eqref{eq:divgrad}, 
\begin{equation}
\label{eq:dJLap2}
\nabla_{\x_i} (J \circ \lambda \circ L^\x) =  2 \sum_{k = (i,j)} g_k \  f'( d^2(\x_i,  \x_j) )  \ (\x_ i - \x_j).  
 \end{equation}
 where $g = \sum_{j=1}^n (Bu_j)^2 \big( \nabla J(\lambda) \big)_j$.
 \end{prop}
\begin{proof} As in the proof of Proposition~\ref{prop:objFunDiffW}, when $J$ be differentiable at $\lambda(L^\x)$, 
the composition  
$J\circ \lambda  \colon {\bf S}^n \to \mathbb R$ is Fr\'echet differentiable with derivative 
$$
 (J\circ \lambda)' (L^\x) = U \  \mathrm{diag} \nabla J (\lambda ) \  U^t
$$
for any diagonalizing matrix $U\in  O^n$ satisfying $L^\x = U  \  \mathrm{diag} \lambda (L^\x) \   U^t$.  
Denoting $V = U \  \mathrm{diag} \nabla J (\lambda) \  U^t$, it follows that
\begin{align}
\label{eq:VDWF}
\nabla_{\x_i}( J \circ \lambda \circ L^\x) &= \langle V , \nabla_{\x_i} (D^\x - W^\x )  \rangle_F \\
\nonumber
&= \sum_k V_{kk} \nabla_{\x_i} D_{kk}^\x - 2 \sum_k V_{ik} \nabla_{\x_i} W_{ik}^\x,
\end{align}
where we used \eqref{eq:dJe}. From
$$
\nabla_{\x_i} D_{kk}^\x = \nabla_{\x_i} W_{ik}^\x + \delta_{ik} \sum_j \nabla_{\x_i} W_{ij}^\x,
$$
we then have that 
\begin{align*}
\nabla_{\x_i}( J \circ \lambda \circ L^\x) &=  \sum_k V_{kk} \nabla_{\x_i} W_{ik}^\x + \sum_k V_{kk}  \delta_{ik} \sum_j \nabla_{\x_i} W_{ij}^\x - 2 \sum_k V_{ik} \nabla_{\x_i} W_{ik}^\x \\
&=  \sum_k V_{kk} \nabla_{\x_i} W_{ik}^\x +  V_{ii} \sum_j \nabla_{\x_i} W_{ij}^\x - 2 \sum_k V_{ik} \nabla_{\x_i} W_{ik}^\x \\
&=  \sum_k (V_{kk} - 2 V_{ik} + V_{ii}) \nabla_{\x_i}  W_{ik}^\x. 
\end{align*}
Equation \eqref{eq:dJLap} now follows from \eqref{eq:dwdx}.

From \eqref{eq:VDWF}, we could also write 
\begin{align}
\nonumber
\nabla_{\x_i} (J\circ \lambda \circ L) &= \langle V, \nabla_{\x_i}  (B^t \text{diag}(w^\x) B)  \rangle_F \\
\nonumber
&= \langle V,  B^t \text{diag}( \nabla_{\x_i}   w) B  \rangle_F \\
\nonumber
& = \langle (BU) \  \mathrm{diag} \nabla J (\lambda ) \  (BU)^t, \text{diag}( \nabla_{\x_i}   w^\x)  \rangle_F \\
\label{eq:VDWF2}
& = \Big\langle g , \nabla_{\x_i}   w^\x  \Big\rangle,
\end{align}
where 
$$
g = \text{diag} \left(  (BU) \  \mathrm{diag} \nabla J (\lambda ) \  (BU)^t \right).
$$ 
Thus, $g\in \mathbb R^{\binom n 2}$ has entries given by
$$
g_k = \sum_{j=1}^n (BU)_{k,j}^2 \big( \nabla J(\lambda) \big)_j. 
$$
Finally, we compute
$$
\nabla_{\x_i} w_k = \begin{cases}
2 f'( d^2(\x_i,  \x_j) )  \ (\x_ i - \x_j) & k = (i,j) \\
0 & \text{otherwise}.
\end{cases}
$$
which gives \eqref{eq:dJLap2}. 
\end{proof}

\bigskip

\noindent {\bf Example.} Consider the spectral function $J(\lambda) = \sum_i \lambda_i$ so that the objective function is simply $\text{tr} L^\x$. In this case, $V= \text{Id}$ and
$$
\nabla_{\x_i} (J \circ \lambda \circ L^\x) = 4 \sum_k  f'( d^2(\x_i,  \x_k) )  \ (\x_ i - \x_k) = \nabla_{\x_i} \sum_{j,k} f( d^2(\x_j,  \x_k) ). 
$$

\begin{rem} Equation \eqref{eq:VDWF2} in the proof of Proposition~\ref{prop:objFunDiffL} implies that the variation of a spectral function of the graph Laplacian with respect to the edge weights is generally given by 
\begin{equation}
\label{eq:edgeGradient}
\nabla_w [J \circ \lambda \circ (B^t \ \text{diag}(w) \ B) ] =  \text{diag}[ (BU) \ \text{diag} \big( \nabla J(\lambda) \big) \ (BU)^t ].
\end{equation}
For example, the gradient of the trace, $ \text{tr} L(w)$, is given by  
$$
\frac{\delta \text{tr} L}{\delta w}  = \text{diag} (B B^t) = 2 \cdot 1_{\binom n 2}.
$$
The factor of two simply indicates that each edge weight effects the degree of two vertices. 
When  $\lambda_2 \big( L(w)\big)$ is a simple eigenvalue,  the gradient is given by 
$$
\frac{\delta \lambda_2}{\delta w} = (B u_2)^2, 
$$
which agrees with the formulas in \cite{Ghosh:2006b,Ghosh:2006}. 
Finally, the gradient of the total effective resistance, $R_{tot} = \sum_j \lambda_j^{-1}$, is computed
$$
\frac{\delta R_{tot}}{\delta w} = - \text{diag} \big( B U \text{diag}(\lambda^{-2} ) (BU)^t \big)  
= - \text{diag}( B (L^\dag)^2 B^t), 
$$
which can also be found in  \cite{ghosh2008}.
\end{rem}

For a general spectral function $J$, the gradients computed in Propositions \ref{prop:objFunDiffW} and \ref{prop:objFunDiffL} can be used together with a quasi-Newton optimization method to efficiently search for a locally optimal pointset configuration; see Section \ref{sec:tori}. 

\section{Parameterization of Lattices} \label{sec:LattParam}
Let $B = [b_1, \ldots, b_n] \in \mathbb R^{n\times n}$ have linearly independent columns. The lattice generated by the basis $B$ is the set  of integer linear combinations of the columns of $B$, 
$$
\mathcal L(B) = \{ Bx \colon x \in \mathbb Z^n \}. 
$$
Let $B$ and $C$ be two lattice bases. We recall that $\mathcal L(B) = \mathcal L(C)$ if and only if there is a unimodular\footnote{A  matrix $A \in \mathbb Z^{n\times n}$ is \emph{unimodular} if $\mathrm{det} A  = \pm 1$.}  
matrix $U$ such that $B = CU$. Thus, there is a one-to-one correspondence between the unimodular $2\times 2$ matrices and the bases of a two-dimensional lattice. 

We say that two lattices are isometric if there is a rigid transformation that maps one to the other. The following proposition parameterizes the space of two-dimensional, unit-volume lattices modulo isometry. 
\begin{prop} \label{prop:LatticeParam}
Every two-dimensional lattice with volume one is isometric to a lattice parameterized by the basis 
$$ \begin{pmatrix} \frac{1}{\sqrt b} & \frac{a}{\sqrt b} \\ 0 & \sqrt b \end{pmatrix}, $$ 
where the parameters $a$ and $b$ are constrained to the set
$$
U := \left\{ (a,b) \in \mathbb R^2 \colon b>0, \ a \in [ 0,1/2  ], \ \text{and} \ a^2 + b^2 \geq 1 \right\}.
$$ 
\end{prop}

The set $U$ defined in Proposition~\ref{prop:LatticeParam} is illustrated  in Figure~\ref{fig:FundDom}.  

\begin{figure}[t]
\begin{center}
\begin{tikzpicture}[scale=.4,thick,>=stealth',dot/.style = {fill = black,circle,inner sep = 0pt,minimum size = 4pt}]

 \filldraw[fill=blue!6]  (0,10) arc(90:60:10cm)  -- (5,17) [dashed] -- (0,17) -- cycle;

\draw[->] (-3,0) -- (13,0) coordinate[label = {below:$a$}];
\draw[->] (0,-.1) coordinate[label={below:$0$}] -- (0,17) coordinate[label = {left:$b$}];

\draw (10,.1) -- (10,-.1) coordinate[label={below:$1$}]; 
\draw (5,-.1)coordinate[label = {below:$0.5$}] -- (5,17) coordinate(top); 
\draw (10,0) arc (0:60:10) arc (60:90:10) node[midway]{\begin{tabular}{c} rhombic \\lattices\end{tabular}}{} arc(90:110:10) ; 
 
 \draw (5,8.660254) node[dot,label={right: \begin{tabular}{c} triangular \\lattice\end{tabular}}](triLat){};
 \draw (0,10) -- (0,13) coordinate[label={center:\rotatebox{90}{\begin{tabular}{c}rectangular\\lattices\end{tabular}}}]{} --(0,17) ;
 \draw (2.5,12) node[label={\begin{tabular}{c} oblique \\lattices\end{tabular}}]{}; 
 \draw (0,10) node[dot,label={below left:\begin{tabular}{c} square \\lattice\end{tabular}}]{};

\end{tikzpicture}
\caption{The set $U$ in Proposition~\ref{prop:LatticeParam}. Parameters $(a,b)$ corresponding to  square, triangular, rectangular, rhombic, and oblique lattices are also indicated.}
\label{fig:FundDom}
\end{center}
\end{figure}
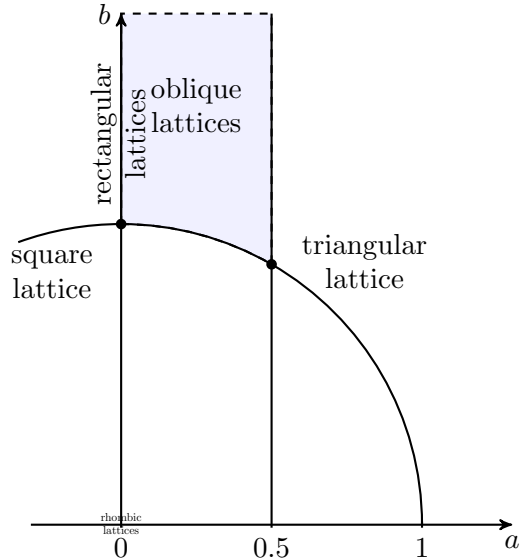

\begin{proof} Consider an arbitrary lattice with unit volume. 
We first choose the basis vectors so that the angle between them is acute.  After a suitable rotation and reflection, we can let the shorter basis vector (with length $\frac{1}{\sqrt b}$) be parallel with the $x$ axis and the longer basis  vector (with length $ \sqrt{ \frac{a^2}{b} + b} = \sqrt{\frac{1}{b} \left( a^2 + b^2 \right)} \geq \sqrt{ \frac{1}{b}}$) lie in the first quadrant. 
Multiplying on the right by a unimodular matrix, $ \begin{pmatrix} 1 & 1 \\ 0 & 1 \end{pmatrix}$, we compute 
$$
\begin{pmatrix} \frac{1}{\sqrt b} & \frac{a}{\sqrt b} \\ 0 & \sqrt b \end{pmatrix} 
 \begin{pmatrix} 1 & 1 \\ 0 & 1 \end{pmatrix} = 
\begin{pmatrix} \frac{1}{\sqrt b} & \frac{a+1}{\sqrt b} \\ 0 & \sqrt b \end{pmatrix} . 
$$
Since this is equivalent to taking $a\mapsto a+1$, it follows that we can identify the lattices associated to the points $(a,b)$ and $(a+1,b)$. 
Thus, we can  restrict the parameter $a$ to the interval 
$\left[ 0, 1/2  \right]$ by symmetry.  For a complete picture of this restriction and how the symmetry naturally arises, see \cite[Proposition $3.2$ and Figure $3$]{KLO2015}.  
\end{proof}

\printbibliography

\end{document}